\documentclass[11pt,a4paper]{article}

\usepackage{amsmath}
\usepackage{amsthm}
\usepackage{amssymb}
\usepackage{latexsym}
\usepackage{textcomp}
\usepackage[T1]{fontenc}
\usepackage[sc]{mathpazo}
\usepackage{titlesec}
\usepackage[hang,flushmargin,bottom]{footmisc}
\usepackage[small,it]{caption}
\usepackage{tikz}
\usetikzlibrary{decorations.pathreplacing}
\usepackage{tabularx}

\usepackage{tcolorbox}
\tcbuselibrary{many}
\newtcolorbox{notebox}[1]{%
    tikznode boxed title,
    boxrule=0.3mm,
    enhanced,
    arc=2mm,
    interior style={white},
    attach boxed title to top center= {yshift=-\tcboxedtitleheight/2},
    fonttitle=\bfseries,
    colbacktitle=white,coltitle=black,
    boxed title style={size=normal,colframe=white,boxrule=0pt},
    title={#1}}
\usepackage{stmaryrd}
\SetSymbolFont{stmry}{bold}{U}{stmry}{m}{n}

\definecolor{lightgray}{rgb}{0.9, 0.9, 0.9}
\definecolor{darkgray}{rgb}{0.3, 0.3, 0.3}
\definecolor{darkblue}{rgb}{0, 0, .4}

\usepackage[bookmarks]{hyperref}
\hypersetup{
	colorlinks=true,
	linkcolor=darkblue,
	anchorcolor=darkblue,
	citecolor=darkblue,
	urlcolor=darkblue,
	pdfpagemode=UseThumbs,
	pdftitle={New bounds on the growth rate of 1324-avoiders},
	pdfsubject={Combinatorics},
	pdfauthor={David Bevan, Robert Brignall, Andrew Elvey~Price, Jay Pantone},
	pdfkeywords={todo}
}
\theoremstyle{plain}
\newtheorem{theorem}{Theorem}[section]
\newtheorem{cor}[theorem]{Corollary}
\newtheorem{lemma}[theorem]{Lemma}

\newtheorem{prop}[theorem]{Proposition}

\newtheorem{prob}[theorem]{Problem}
\theoremstyle{definition}

\theoremstyle{remark}

\newcounter{todocounter}

\makeatletter
\newfont{\footsc}{cmcsc10 at 8truept}
\newfont{\footbf}{cmbx10 at 8truept}
\newfont{\footrm}{cmr10 at 10truept}
\newcolumntype{C}[1]{>{\centering\let\newline\\\arraybackslash\hspace{0pt}}m{#1}}

\title{A structural characterisation of Av(1324) \\
and new bounds on its growth rate}
\author{%
	\begin{tabular}{C{7.25cm}C{7.25cm}}
		David Bevan\newline
    	\small University of Strathclyde\newline
	    \small Glasgow, UK
		&
		Robert Brignall\newline
		\small The Open University\newline
		\small Milton Keynes, UK
		\\\ \\[-5pt]
		Andrew Elvey Price\footnote{The third author was supported by an Australian government research training program scholarship.}\newline		
		\small LaBRI,
               Universit\'e de Bordeaux\newline
		\small Talence, France
		&
		Jay Pantone\newline
		\small Marquette University\newline
		\small Milwaukee, WI, USA
	\end{tabular}%
}%

\date{}

\newcommand{\ceil}[1]{\left\lceil #1 \right\rceil}
\newcommand{\floor}[1]{\left\lfloor #1 \right\rfloor}

\DeclareMathOperator{\av}{Av}

\newcommand{\grid}{\mathrm{Grid}}
\newcommand{\gr}{\mathrm{gr}}
\newcommand{\A}{\mathcal{A}}
\newcommand{\B}{\mathcal{B}}

\newcommand{\C}{\mathcal{C}}
\newcommand{\D}{\mathcal{D}}

\newcommand{\liminfty}[1][n]{\lim\limits_{#1\rightarrow\infty}}
\newcommand{\limsupinfty}[1][n]{\limsup\limits_{#1\rightarrow\infty}}
\newcommand{\liminfinfty}[1][n]{\liminf\limits_{#1\rightarrow\infty}}
\newcommand{\LLL}{\mathcal{L}}

\newcommand{\PPP}{\mathcal{P}}

\newcommand{\varempty}{\scalebox{1.2}{$\varnothing$}}
\newcommand{\varempsm}{\scalebox{0.9}{$\varnothing$}}
\newcommand{\veps}{\varepsilon}
\newcommand{\range}[1]{[#1]}

\newcommand{\leqs}{\leqslant}
\newcommand{\geqs}{\geqslant}

\newcommand{\discrim}{\operatorname{discrim}}
\newcommand{\Res}{\operatorname{Res}}

\newcommand{\+}{\hspace{0.07 em}}

\newcommand{\plotptradius}{0.275}
\newcommand{\setplotptradius}[1]{\renewcommand{\plotptradius}{#1}}

\newcommand{\plotptfun}{\fill}
\newcommand{\setplotptfun}[1]{\renewcommand{\plotptfun}{#1}}

\newcommand{\plotpt}[3][] 
{ \plotptfun[#1,radius=\plotptradius] (#2,#3) circle; }

\newcommand{\plotpermnobox}[3][]  
{
  \foreach \y [count=\x] in {#3}
  {
    \ifnum0=\y {} \else {
      \plotpt[#1]{\x}{\y}
    } \fi
  }
}

\newcommand{\ringpoint}[2][]  
{
  \fill[#1,radius=0.3] (#2) circle;
  \draw[#1,radius=0.4] (#2) circle;
}

\newcommand{\plotvert}[3][]  
{
  \foreach \y in {#3}
  {
    \ifnum0=\y {} \else {
      \plotpt[#1]{#2}{\y}
    } \fi
  }
}

\newenvironment{smallmx}[1][{}] {\left(\!\begin{smallmatrix}} {\end{smallmatrix}\!\right)}
\newenvironment{gridmx}[1][{}] {\grid^{#1}\!\begin{smallmx}} {\end{smallmx}}

\begin{document}
\maketitle
\begin{abstract}
\noindent
We establish an improved lower bound of 10.271 for the exponential growth rate of the class of permutations avoiding the pattern 1324, and an improved upper bound of 13.5.
These results depend on a new exact structural characterisation of 1324-avoiders as a subclass of an infinite staircase grid class, together with precise asymptotics of a small \emph{domino} subclass whose enumeration we relate to West-two-stack-sortable permutations and planar maps. The bounds are established by carefully combining copies of the dominoes in particular ways consistent with the structural characterisation.
The lower bound depends on
concentration results concerning the substructure of a typical domino,
the determination of exactly when
dominoes can be combined in the fewest distinct ways,
and technical analysis of the resulting generating function.
\end{abstract}

%
%
%
%
%
%
%
%
%
%
%
%
\section{Introduction}

The class of 1324-avoiding permutations is notoriously difficult to enumerate.
The other permutation classes that avoid a single permutation of length 4 were enumerated explicitly in the 1990s (see B\'ona~\cite{bona:exact-enumerati:} and Gessel~\cite{gessel:symmetric-funct:}).
In contrast, even the exponential growth rate of $\av(1324)$
remains to be determined exactly.

If $\sigma=\sigma(1)\ldots\sigma(n)$ is a permutation of length $n$, written in one-line notation, and
$\pi$ is a permutation of length $k\leqs n$, then we
say that $\pi$ is \emph{contained} in $\sigma$ if there is a subsequence $i_1, \ldots, i_k$ of $1,\dots,n$
such that
$\pi(\ell) < \pi(m)$ if and only if $\sigma(i_\ell) < \sigma(i_m)$, for all
$\ell,m\in\range{k}$, that is $\sigma(i_1)\ldots\sigma(i_k)$ is \emph{order isomorphic} to $\pi$.
We say that $\sigma(i_1)\ldots\sigma(i_k)$ is an \emph{occurrence} of $\pi$ in $\sigma$
and, for each $\ell\in\range{k}$, that $\sigma(i_\ell)$ \emph{acts as a $\pi(\ell)$} in this occurrence.
For example, 425 is the only occurrence of 213 in 84672531; the entry 5 acts as a 3 in this occurrence of 213.

If $\pi$ is not contained in $\sigma$, then $\sigma$ \emph{avoids} $\pi$.
We use $\av(\pi)$ to denote the set consisting of all permutations that avoid $\pi$. Note that $\av(\pi)$ is a hereditary class, or \emph{permutation class}, in the sense that whenever $\sigma \in \av(\pi)$ and $\tau$ is contained in $\sigma$, then $\tau \in \av(\pi)$.

The exponential \emph{growth rate} of the class $\av(\pi)$ is
\[
\gr(\av(\pi)) \;=\; \lim_{n\to\infty} \sqrt[n]{\big|\av_n(\pi)\big|},
\]
where $\av_n(\pi)$ denotes the set of permutations of length $n$ that avoid $\pi$.
This limit is known to exist as a consequence of the resolution of the Stanley-Wilf conjecture by Marcus and Tardos~\cite{marcus:excluded-permut:}.
More generally, if $\A$ is an infinite set of combinatorial objects, then the \emph{growth rate} of~$\A$ is $\gr(\A)=\limsup_{n\to\infty}\sqrt[n]{|\A_n|}$,
where we use $\A_n$ to denote the set of elements of $\A$ with size~$n$.

For an introduction to the enumerative theory of permutation classes, see Vatter's thorough exposition~\cite{VatterSurvey}.
The topic is also presented
in a broader context in the books by
B\'ona~\cite{Bona2012} and
Kitaev~\cite{Kitaev2011}.

\begin{table}[ht]
\centering
\begin{tabular}{l|c|c}
&\textbf{Lower}&\textbf{Upper}                                                                                          \\
\hline
2004: B\'ona~\cite{bona:the-exponential-up:}                                    &                           &           288\phantom{.00}  \\  
2005: B\'ona~\cite{bona:the-limit-of-a-:}                                       & \phantom{0}9\phantom{.00}                               \\  
2006: Albert et al.~\cite{albert:on-the-wilf-sta:}                              & \phantom{0}9.47                                         \\  
2012: Claesson, Jel\'inek and Steingr\'imsson~\cite{claesson:upper-bounds-fo:}  &                           & \phantom{0}16\phantom{.00}  \\  
2014: B\'ona~\cite{bona:a-new-upper-bou:}                                       &                           & \phantom{0}13.93            \\  
2015: B\'ona~\cite{bona:a-new-record-for:}                                      &                           & \phantom{0}13.74            \\  
2015: Bevan~\cite{BevanAv1324GR}                                                & \phantom{0}9.81                                         \\  
\phantom{2017: }\emph{This work}                                                &            10.27           & \phantom{0}13.5\phantom{0}  \\
\end{tabular}
\caption{A chronology of lower and upper bounds for $\gr(\av(1324))$}\label{table-chron}
\end{table}

Our interest is in the growth rate of the class $\av(1324)$, the subject of a number of papers over the last decade and a half.
For an entertaining essay placing the problem in a wider historical context, see~\cite{Egge2015}.
The history of rigorous lower and upper bounds for $\gr(\av(1324))$ is summarised in Table~\ref{table-chron}. In addition to these, Claesson, Jel\'inek and Steingr\'imsson~\cite{claesson:upper-bounds-fo:} make a conjecture regarding the number of 1324-avoiders of each length that have a fixed number of inversions, which if proven would yield an improved upper bound of $e^{\pi\sqrt{2/3}}\approx 13.002$.

With the help of computers, $|\av_n(1324)|$ has been determined for all $n\leqs50$.
Conway, Gutt\-mann and Zinn-Justin~\cite{conway:on-1324-avoiding:,CGZJ2018} have analysed the numbers and give
a numerical estimate for $\gr(\av(1324))$ of $\mu\approx11.600\pm 0.003$. They also conjecture that $|\av_n(1324)|$ behaves asymptotically as $A\!\cdot\! \mu^n\!\cdot\!\lambda^{\sqrt{n}}\!\cdot\! n^\alpha$, for certain estimated constants $A$, $\lambda$ and $\alpha$.
If this conjecture were proved, then as a consequence of~\cite[Theorem~9]{GP2017},
it would imply that the counting sequence for 1324-avoiders is not P-recursive (i.e.~does not satisfy a linear recurrence
with polynomial coefficients), perhaps going some way to explain the difficulties faced in its enumeration.

Our contribution to the investigation of the 1324-avoiders
is to establish new rigorous lower and upper bounds on $\gr(\av(1324))$.
These rely on a new structural characterisation of $\av(1324)$ as a subclass of an infinite staircase grid class, which we present in the next section.
In Section~\ref{sec-dominoes}, we investigate pairs of adjacent cells in the staircase, which we call \emph{dominoes}, and give an exact enumeration (Theorem~\ref{thm-twocell}).
Together with a result concerning \emph{balanced} dominoes, this is sufficient to deduce
a new upper bound of 13.5 and a new lower bound of 10.125 on the growth rate of $\av(1324)$,
which we present in the following two sections as Theorems~\ref{thmUpperBound} and~\ref{thmLowerBound1}.

The lower bound can be increased by investigating the structure of dominoes in greater detail.
In Section~\ref{secLeavesAndStrips}, we prove two asymptotic concentration results, relating to \emph{leaves} and \emph{empty strips}.
Section~\ref{secLowerBound2} then presents a refinement of our staircase construction,
a lower bound on the number of ways of combining dominoes,
and a technical analysis of the resulting generating function. This yields, in Theorem~\ref{thmLowerBound2}, a lower bound on $\gr(\av(1324))$ of 10.271.

%
%
%
%
%
%
%
%
%
%
%
%
\section{Staircase structure}\label{secStaircase}

In this section, we present a structural characterisation of $\av(1324)$ as a subclass of a larger permutation class. This class 
is a
\emph{staircase class}, which is a special case of an infinite grid class of permutations.
We begin by defining finite and infinite grid classes.

Suppose that $M$ is a $t\times u$ matrix of (possibly empty) permutation classes, where $t$ is the number of columns and $u$ the number of rows.
An \emph{$M$-gridding} of a permutation $\sigma$ of length $n$ is a pair of sequences
$1 = c_1 \leqs \ldots \leqs c_{t+1} = n+1$ (the \emph{column dividers}) and
$1 = r_1 \leqs \ldots \leqs r_{u+1} = n+1$ (the \emph{row dividers}) such that
for all $k\in\range{t}$ and $\ell\in\range{u}$, the entries of $\sigma$ whose indices are in $[c_k,c_{k+1})$ and values in $[r_\ell,r_{\ell+1})$
are order isomorphic to an element of $M_{k,\ell}$.
Thus, an $M$-gridding of $\sigma$ partitions the entries of $\sigma$, with one part for each cell in $M$.
A permutation together with one of its $M$-griddings is called an \emph{$M$-gridded permutation}.

The \emph{grid class} of $M$, denoted $\grid(M)$, consists of all the permutations that have an $M$-gridding.
We also use $\grid^\#(M)$ to denote the set of all $M$-gridded permutations, every permutation in
$\grid(M)$ being present once with each of its $M$-griddings. 

The definition of a grid class extends naturally for infinite matrices.
If $M$ is an infinite matrix of permutation classes, then the \emph{infinite grid class} $\grid(M)$ consists of all the permutations that have an $M'$-gridding, for some finite submatrix $M'$ of $M$.

Of direct interest to us are \emph{staircase classes}, infinite grid classes that have a staircase structure (for more on staircase classes, see~\cite{albert:on-the-growth-of-merges:}).
Given two permutation classes, $\C$ and $\D$,
the \emph{descending $(\C,\D)$ staircase} is the infinite grid class
  \[
  \begin{gridmx}
  \C\,      &    &        & \varempsm    \\[2pt]
  \D\,      & \C                         \\[2pt]
            & \D & \C     &              \\[-4pt]
  \varempsm &    & \ddots & \!\ddots
  \end{gridmx}\!,
  \]
in which $\C$ occurs in each cell on the diagonal, $\D$ occurs on the subdiagonal, and the remaining cells contain the empty permutation class $\varempty$.

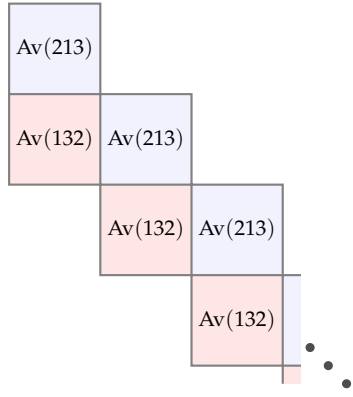
\begin{figure}[ht]
\begin{center}
	\begin{tikzpicture}[scale=1.2]
		\foreach \x in {1,2,3}
		{
          \fill[blue!5] (\x+1,4-\x) rectangle (\x+2,3-\x);
          \fill[red!10] (\x+1,3-\x) rectangle (\x+2,2-\x);
		  \node at (\x+1.5, 2.5-\x) {\scriptsize $\av(132)$};
		  \node at (\x+1.5, 3.5-\x) {\scriptsize $\av(213)$};
		}
        \fill[blue!5] (5,-1) rectangle (5.2,0);
        \fill[red!10] (5,-1.2) rectangle (5.2,-1);
		\draw [thick, gray] (2,2)--(2,3)--(3,3)--(3,1)--(2,1)--(2,2)--(4,2)--(4,0)--(3,0)--(3,1)--(5,1)--(5,-1)--(4,-1)--(4,0)--(5.2,0);
        \draw [thick, gray] (5,-1.2)--(5,-1)--(5.2,-1);
        \fill[radius=0.05,darkgray] (5.3,-.8) circle;
        \fill[radius=0.05,darkgray] (5.5,-1) circle;
        \fill[radius=0.05,darkgray] (5.7,-1.2) circle;
	\end{tikzpicture}
	\caption{The descending $\big(\!\av(213),\av(132)\big)$ staircase 
containing $\av(1324)$}
    \label{fig-staircase}
\end{center}
\end{figure}

The class of 1324-avoiders is a subclass of the descending $\big(\!\av(213),\av(132)\big)$ staircase. This staircase class is central to our analysis, and we call it simply \emph{the staircase}. It is illustrated in Figure~\ref{fig-staircase}.

Later, we make use of an important property of the cells in the staircase, which we introduce now.
The \emph{skew sum} of two permutations, denoted \mbox{$\sigma\ominus\tau$}, consists of a copy of $\sigma$ positioned to the upper left of a copy of $\tau$.
Formally,
given two permutations $\sigma$ and $\tau$ with lengths $k$ and $\ell$ respectively, their {skew sum} is the permutation of length $k+\ell$ consisting of a shifted copy of $\sigma$ followed by $\tau$:
$$
(\sigma\ominus\tau)(i) \;=\;
\begin{cases}
  \ell+\sigma(i)   & \text{if~} 1\leqs i\leqs k , \\
  \tau(i-k)        & \text{if~} k+1 \leqs i\leqs k+\ell .
\end{cases}
$$
A permutation is \emph{skew indecomposable} if it cannot
be expressed as the skew sum of two shorter permutations. Note that every permutation
has a unique representation as the skew sum of a sequence of skew indecomposable components.
This representation is known as its \emph{skew decomposition}.
The permutation classes
$\av(213)$ and $\av(132)$, used in the staircase,
are both \emph{skew closed}, in the sense that $\sigma\ominus\tau$ is in the class if both $\sigma$ and $\tau$ are.
The permutations in a skew closed class are precisely the skew sums of sequences of the skew indecomposable permutations in the class.

\begin{prop}\label{prop-staircase}
  $\av(1324)$ is contained in the descending $\big(\!\av(213),\av(132)\big)$ staircase.
\end{prop}

To prove this result, we describe how to construct an explicit gridding of any 1324-avoider in the staircase.
Here, and elsewhere in our discussion, we identify a permutation $\sigma$ with its \emph{plot}, the set of points $(i,\sigma(i))$ in the Euclidean plane, and refer to its entries as \emph{points}.

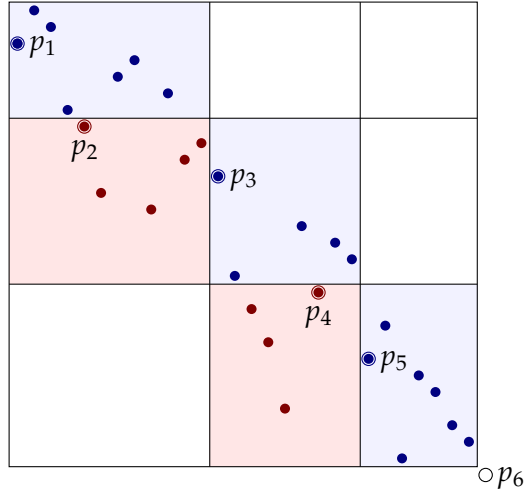
\begin{figure}[ht]
  \begin{center}
    \begin{tikzpicture}[scale=0.22]
      \fill[blue!5] (0.5,30.5) rectangle (12.5,23.5);
      \fill[red!10] (0.5,23.5) rectangle (12.5,13.5);
      \fill[blue!5] (12.5,23.5) rectangle (21.5,13.5);
      \fill[red!10] (12.5,13.5) rectangle (21.5,2.5);
      \fill[blue!5] (21.5,13.5) rectangle (28.5,2.5);
      \draw[thin] (0.5,2.5) rectangle (28.5,30.5);
      \setplotptradius{0.3}
      \plotpermnobox[blue!50!black]{}{28, 30, 29, 24,  0,  0, 26, 27,  0, 25,  0,  0,  0, 14,  0,  0, 0, 17,  0, 16, 15, 0, 11, 3, 8, 7, 5, 4}
      \plotpermnobox[red!50!black]{} { 0,  0,  0,  0,  0, 19,  0,  0, 18,  0, 21, 22,  0,  0, 12, 10, 6,  0,  0,  0,  0, 0,  0, 0, 0, 0, 0, 0}
      \ringpoint[blue!50!black]{1,28}
      \node[] at (2.5,27.8) {$p_1$};
      \ringpoint[red!50!black]{5,23}
      \draw[thin] (0.5,23.5) -- (28.5,23.5);
      \node[] at (5,21.5) {$p_2$};
      \ringpoint[blue!50!black]{13,20}
      \draw[thin] (12.5,2.5) -- (12.5,30.5);
      \node[] at (14.5,19.8) {$p_3$};
      \ringpoint[red!50!black]{19,13}
      \draw[thin] (0.5,13.5) -- (28.5,13.5);
      \node[] at (19,11.5) {$p_4$};
      \ringpoint[blue!50!black]{22,9}
      \draw[thin] (21.5,2.5) -- (21.5,30.5);
      \node[] at (23.5,8.8) {$p_5$};
      \draw[radius=0.4] (29,2) circle;
      \node[] at (30.5,1.8) {$p_6$};
    \end{tikzpicture}
	\caption{The greedy gridding of a 1324-avoider in the staircase}
    \label{fig-greedy}
  \end{center}
\end{figure}

\begin{proof}
  Consider any $\sigma\in\av(1324)$ of length $n$. We construct a gridding of $\sigma$ in the staircase as follows.
  Let $p_1$ be the leftmost point of $\sigma$, and
  iteratively identify subsequent points $p_2,\ldots,p_k$
  as follows.
  See Figure~\ref{fig-greedy} for an illustration.
  \vspace{-9pt}\begin{itemize}\itemsep0pt
    \item If $i$ is even, let $p_i$ be the
        uppermost point of $\sigma$ that 
        acts as a 1
        in an occurrence of 213 consisting only of points 
        to the right of the column divider adjacent to
        $p_{i-1}$.
        Insert a row divider immediately above $p_i$.
        If no suitable point exists, terminate.
    \item If $i>1$ is odd, let $p_i$ be the leftmost point of $\sigma$ that 
        acts as a 2
        in an occurrence of 132 consisting only of points 
        below the row divider adjacent to
        $p_{i-1}$.
        Insert a column divider immediately to the left of $p_i$.
        If no suitable point exists, terminate.
  \end{itemize}\vspace{-9pt}
  Since three points are required for an occurrence of 213 or 132, each cell (except possibly the last) contains at least two points.
  So this process terminates after identifying $k$ points, where $k\leqs\ceil{n/2}$.
  Finally, let $p_{k+1}$ be a virtual point at $(n+1,0)$, below and to the right of all points of $\sigma$.

  By construction, if $i\in[2,k+1]$ is even, then the points of $\sigma$ above $p_i$ and 
  to the right of the column divider adjacent to
  $p_{i-1}$ avoid 213.
  Analogously, if $i\in[3,k+1]$ is odd, then the points of $\sigma$ to the left of $p_i$ and 
  below the row divider adjacent to
  $p_{i-1}$ avoid 132.

  Furthermore, if $i\in[2,k]$ is even, then there are no points of $\sigma$ below $p_i$ and to the left of $p_{i-1}$, since any such point would form a 1324 with the 213 of which $p_i$ acts as a 1.
  Analogously, if $i\in[3,k]$ is odd, then there are no points of $\sigma$ to the right of $p_i$ and above $p_{i-1}$, since any such point would form a 1324 with the 132 of which $p_i$ acts as a 2.

  Thus, the column and row dividers specify a valid $M$-gridding of $\sigma$, where $M$ is a finite submatrix of the infinite matrix defining the staircase.
\end{proof}

\begin{figure}[ht]
  \centering
  \includegraphics[scale=0.3]{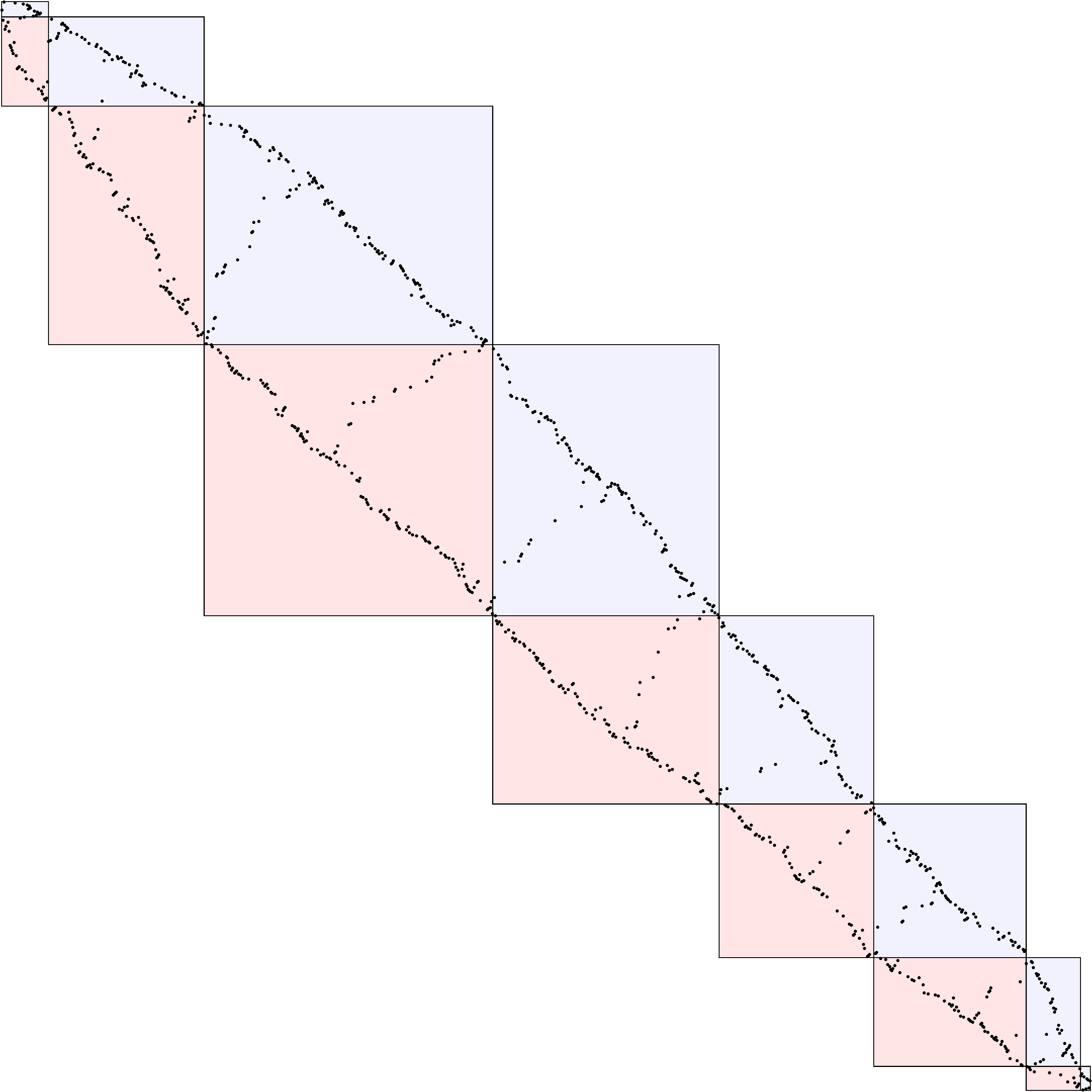}
  \caption{The greedy gridding of a 1324-avoider of length 1000}\label{figEinar}
\end{figure}

We call the gridding of a 1324-avoider $\sigma$
constructed in the proof of Proposition~\ref{prop-staircase}
the
\emph{greedy gridding of $\sigma$}, because, as we descend the staircase, we place as
many points of $\sigma$ as possible in each subsequent
cell.
See Figure~\ref{figEinar} for an illustration of the greedy gridding of a large permutation.\footnote{The data for Figure~\ref{figEinar} was provided by Einar Steingr\'imsson from the investigations he describes in~\cite[Footnote~4]{Steingrimsson2013}. It was generated using a Markov chain Monte Carlo process from~\cite{ML2010}.}

This structural characterisation has not been presented previously.
However, the colouring approach used by Claesson, Jel\'inek and Steingr\'imsson in~\cite{claesson:upper-bounds-fo:} and refined by B\'ona in~\cite{bona:a-new-upper-bou:,bona:a-new-record-for:} depends on the fact that
$\av(1324)$ is a subclass of the \emph{merge} of the permutation classes $\av(213)$ and $\av(132)$.
Given two permutation classes $\C$ and $\D$, their merge, written $\C\odot\D$, is the set of all permutations
whose entries can be coloured blue and red so that the blue subsequence is order isomorphic to a
member of $\C$ and the red subsequence is order isomorphic to a member of $\D$.

The descending staircase is contained in the merge $\av(213)\odot\av(132)$, since points gridded in the upper, $\av(213)$, cells collectively avoid 213, and the remaining points gridded in the lower, $\av(132)$, cells collectively avoid 132.
Thus our new characterisation is a refinement of that used previously.
However, the growth rate of the staircase and that of the merge are both 16 (see~\cite{albert:on-the-growth-of-merges:}), so
Proposition~\ref{prop-staircase} doesn't immediately yield any improvement over the upper bound in~\cite{claesson:upper-bounds-fo:}.

%
%
%
%
%
%
%
%
%
%
%
%
\section{1324-avoiding dominoes}\label{sec-dominoes}

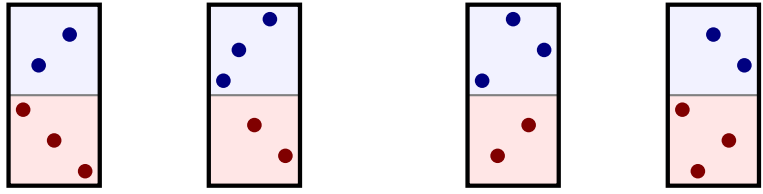
\begin{figure}[ht]
\begin{center}
   \begin{tikzpicture}[scale=1.2]
    \fill[blue!5] (0,1) rectangle (1,2);
    \fill[red!10] (0,0) rectangle (1,1);
   	\draw[thick, gray] (0,1) -- (1,1);
   	\draw[ultra thick] (0,0) rectangle (1,2);
                \fill[radius=0.08,red!50!black] (.16,.84) circle;
                \fill[radius=0.08,blue!50!black] (.33,1.33) circle;
                \fill[radius=0.08,red!50!black] (.5,.5) circle;
                \fill[radius=0.08,blue!50!black] (.67,1.67) circle;
                \fill[radius=0.08,red!50!black] (.84,.16) circle;
   \end{tikzpicture}
   $\qquad\quad$
   \begin{tikzpicture}[scale=1.2]
    \fill[blue!5] (0,1) rectangle (1,2);
    \fill[red!10] (0,0) rectangle (1,1);
   	\draw[thick, gray] (0,1) -- (1,1);
   	\draw[ultra thick] (0,0) rectangle (1,2);
                \fill[radius=0.08,blue!50!black] (.16,1.16) circle;
                \fill[radius=0.08,blue!50!black] (.33,1.5) circle;
                \fill[radius=0.08,red!50!black] (.5,.67) circle;
                \fill[radius=0.08,blue!50!black] (.67,1.84) circle;
                \fill[radius=0.08,red!50!black] (.84,.33) circle;
   \end{tikzpicture}
   $\qquad\quad\qquad$
   \begin{tikzpicture}[scale=1.2]
    \fill[blue!5] (0,1) rectangle (1,2);
    \fill[red!10] (0,0) rectangle (1,1);
   	\draw[thick, gray] (0,1) -- (1,1);
   	\draw[ultra thick] (0,0) rectangle (1,2);
                \fill[radius=0.08,blue!50!black] (.16,1.16) circle;
                \fill[radius=0.08,red!50!black] (.33,.33) circle;
                \fill[radius=0.08,blue!50!black] (.5,1.84) circle;
                \fill[radius=0.08,red!50!black] (.67,.67) circle;
                \fill[radius=0.08,blue!50!black] (.84,1.5) circle;
   \end{tikzpicture}
   $\qquad\quad$
   \begin{tikzpicture}[scale=1.2]
    \fill[blue!5] (0,1) rectangle (1,2);
    \fill[red!10] (0,0) rectangle (1,1);
   	\draw[thick, gray] (0,1) -- (1,1);
   	\draw[ultra thick] (0,0) rectangle (1,2);
                \fill[radius=0.08,red!50!black] (.16,.84) circle;
                \fill[radius=0.08,red!50!black] (.33,.16) circle;
                \fill[radius=0.08,blue!50!black] (.5,1.67) circle;
                \fill[radius=0.08,red!50!black] (.67,.5) circle;
                \fill[radius=0.08,blue!50!black] (.84,1.33) circle;
   \end{tikzpicture}
   \caption{Four distinct small dominoes}\label{figDominoes}
\end{center}
\end{figure}

To establish bounds on the growth rate of $\av(1324)$, we investigate pairs of adjacent cells in the griddings of 1324-avoiders in the staircase.
We define a 1324-avoiding vertical \emph{domino} to be a two-cell \emph{gridded permutation} in
$\begin{gridmx}[\#]\av(213)\\\av(132)\end{gridmx}$ whose underlying permutation avoids 1324.
See Figure~\ref{figDominoes} for an illustration of four dominoes, the two at the left being distinct griddings of 34251, and the two at the right being distinct griddings of 31524.
Let $\D$ be the set of dominoes.
It is important to note that
\[
  \begin{tikzpicture}[scale=.8]
    \fill[blue!5] (0,1) rectangle (1,2);
    \fill[red!10] (0,0) rectangle (1,1);
   	\draw[thick, gray] (0,1) -- (1,1);
   	\draw[very thick] (0,0) rectangle (1,2);
    \fill[radius=0.08,red!50!black] (.2,.33) circle;
    \fill[radius=0.08,blue!50!black] (.36,1.33) circle;
    \fill[radius=0.08,red!50!black] (.64,.67) circle;
    \fill[radius=0.08,blue!50!black] (.8,1.67) circle;
 \end{tikzpicture}
 \raisebox{21pt}{$\;\;\;\,\notin\;\;\D$,}
\]
since $\D$ consists of gridded 1324-avoiders. Moreover, within the grid class
$\begin{gridmx}\av(213)\\\av(132)\end{gridmx}$,
this is the \emph{only} arrangement of points that must be avoided, since it is the only possible gridding of 1324 in
the two cells. With the cell divider in any other position, either the top cell contains a 213 or the bottom cell contains a 132.

In this section we enumerate the gridded permutations in $\D$ by placing them in bijection with certain arch configurations, proving the following theorem.

\begin{theorem}\label{thm-twocell}
	The number of $n$-point dominoes is $\displaystyle\frac{2(3n+3)!}{(n + 2)!(2n + 3)!}$. Consequently, $\gr(\D)=27/4$.
\end{theorem}

This theorem, along with the result that \emph{balanced} dominoes have the same growth rate (Proposition~\ref{propBalancedDominoes}), gives us enough information to calculate improved upper and lower bounds for the growth rate of $\av(1324)$.

In order to prove Theorem~\ref{thm-twocell}, our first task is to establish a functional equation for the set of dominoes $\D$. We do this by representing dominoes as 
configurations consisting of an interleaved pair of arch systems, one for each of the two cells.

%
%
%
%
\subsection{Arch systems}

Let an $n$-point \emph{arch system} consist of $n$ points on a horizontal line together with zero or more noncrossing arcs, all on the same side of the line,
connecting distinct pairs of points, such that
no point is the left endpoint of more than one arc
and
no point is the right endpoint of more than one arc.
See Figure~\ref{figArchSystems}.
Note that these are not non-crossing \emph{matchings}.

These arch systems are equinumerous with
domino cells.\footnote{Despite being enumerated by the Catalan numbers, these specific arch systems are, rather surprisingly, not included in Stanley's book~\cite{StanleyCatalan}.}
We make use of a bijection in which arcs correspond to
occurrences of 12 in the cells, having the form $k(k+1)$ for some value~$k$.

\begin{figure}[ht]
$$
  \begin{tikzpicture}[scale=0.26,line join=round]
    \fill[blue!5] (.25,.25) rectangle (9.75,9.75);
    \draw[thick] (.25,.25) rectangle (9.75,9.75);
    \foreach \x in {1,...,9} \draw [gray,very thin] (\x,-3.75)--(\x,9.75);
    \draw [gray,thin] (0,-4)--(10,-4);
    \plotpermnobox[blue!50!black]{}{7,0,0,1,3,0,4,0,0}
    \setplotptfun{\draw}
    \plotpermnobox[blue!50!black,thick]{}{0,9,8,0,0,6,0,5,2}
    \setplotptfun{\fill}
    \draw [white,very thick] (1,-4.25) arc [radius=2,start angle=180,end angle=360];  
    \draw [blue!50!black,very thick] (3,-3.75) arc [radius=1,start angle=0,end angle=180];
    \draw [blue!50!black,very thick] (7,-3.75) arc [radius=1,start angle=0,end angle=180];
    \draw [blue!50!black,very thick] (8,-3.75) arc [radius=0.5,start angle=0,end angle=180];
    \draw [blue!50!black,very thick] (9,-3.75) arc [radius=2.5,start angle=0,end angle=180];
    \plotpermnobox[blue!50!black]{}{-4, 0, 0,-4,-4, 0,-4, 0, 0}
    \setplotptfun{\draw}
    \plotpermnobox[blue!50!black,thick]{}{ 0,-4,-4, 0, 0,-4, 0,-4,-4}
    \setplotptfun{\fill}
  \end{tikzpicture}
  \qquad\qquad\qquad\quad
  \begin{tikzpicture}[scale=0.26,line join=round]
    \fill[red!10] (.25,.25) rectangle (13.75,13.75);
    \draw[thick] (.25,.25) rectangle (13.75,13.75);
    \foreach \x in {1,...,13} \draw [gray,very thin] (\x,-1.25)--(\x,13.75);
    \draw [gray,thin] (0,-1.5)--(14,-1.5);
    \plotpermnobox[red!50!black]{}{ 0,12,13, 0,0,0,0,4,7,0,8,0,9}
    \setplotptfun{\draw}
    \plotpermnobox[red!50!black,thick]{}{11, 0, 0,10,6,5,3,0,0,2,0,1,0}
    \setplotptfun{\fill}
    \draw [red!50!black,very thick] (1,-1.75) arc [radius=0.5,start angle=180,end angle=360];
    \draw [red!50!black,very thick] (2,-1.75) arc [radius=0.5,start angle=180,end angle=360];
    \draw [red!50!black,very thick] (5,-1.75) arc [radius=2,start angle=180,end angle=360];
    \draw [red!50!black,very thick] (7,-1.75) arc [radius=0.5,start angle=180,end angle=360];
    \draw [red!50!black,very thick] (9,-1.75) arc [radius=1,start angle=180,end angle=360];
    \draw [red!50!black,very thick] (11,-1.75) arc [radius=1,start angle=180,end angle=360];
    \plotpermnobox[red!50!black]{}{   0,-1.5,-1.5,   0,   0,   0,   0,-1.5,-1.5,   0,-1.5,   0,-1.5}
    \setplotptfun{\draw}
    \plotpermnobox[red!50!black,thick]{}{-1.5,   0,   0,-1.5,-1.5,-1.5,-1.5,   0,   0,-1.5,   0,-1.5,   0}
    \setplotptfun{\fill}
  \end{tikzpicture}
$$
\caption{
A 213-avoider and a 132-avoider
with their arch systems}\label{figArchSystems}
\end{figure}
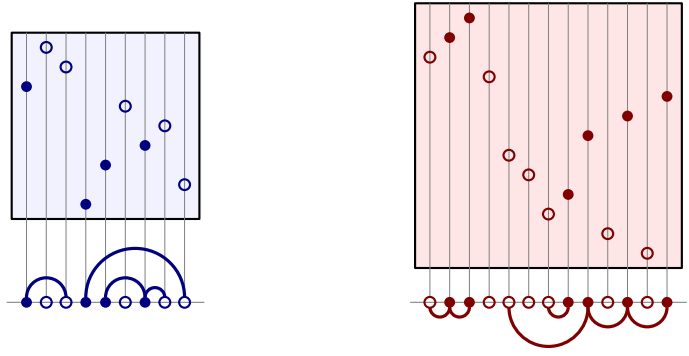

\begin{prop}\label{propArchSystems}
  Both
  $\av_n(213)$ and $\av_n(132)$
  are in bijection with $n$-point arch systems.
\end{prop}

\newcommand{\blueArc}[2] {\draw [blue!50!black,very thick] (#1,1.25) arc (0:180:#2);} 
\newcommand{\fillBlueArc}[2] {\fill [gray!40!white] (#1,1.25) arc (0:180:#2); \draw [gray!40!white,line width=6] (#1,1)--(#1-2*#2,1);} 
\newcommand{\bluePt}[1]  {\draw [blue!50!black,very thick] (#1,1.25)--(#1,1.85);}
\newcommand{\redArc}[2]{\draw [red!50!black,very thick] (#1,0.75) arc (180:360:#2);} 
\newcommand{\redPt}[1] {\draw [red!50!black,very thick] (#1,0.75)--(#1,0.15);}
\newcommand{\archavoid}{
\!
\raisebox{-2pt}{
  \begin{tikzpicture}[scale=0.15,line join=round]
    \draw [gray,thin] (0,1)--(5,1);
    \redArc{1}{1 and .8}
    \blueArc{4}{1 and .8}
    \plotpermnobox[blue!50!black]{}{0,1,0,1}
    \plotpermnobox[red!50!black]{}{1,0,1,0}
  \end{tikzpicture}
}
\!
}

\begin{proof}
  We define a mapping $\Lambda$ from $\av(213)$ and $\av(132)$ to arch systems.
  This mapping is illustrated in Figure~\ref{figArchSystems}.
  Given 
  a 213-avoiding or 132-avoiding
  permutation $\sigma$ of length $n$, let the points of the corresponding arch system $\Lambda(\sigma)$ be positioned at $1,\ldots,n$ on the line.
  For each pair $i,j$ with $1\leqs i<j\leqs n$, connect the points at $i$ and $j$ with an arc if and only if $\sigma(j)=\sigma(i)+1$.

  The result is a valid arch system. Crossing arcs could only result from an occurrence in $\sigma$ of either 1324 or 3142, both of which contain both 213 and 132,
  and by construction no point can be the left endpoint of more than one arc or the right endpoint of more than one arc.

  In the converse direction, we recursively define mappings $\Pi_{213}$ and $\Pi_{132}$ from arch systems to $\av(213)$ and $\av(132)$ respectively,
  such that for any arch system $\alpha$, we have
  \begin{equation}\label{eqLambdaPi}
    \Lambda(\Pi_{213}(\alpha)) \;=\; \Lambda(\Pi_{132}(\alpha)) \;=\; \alpha .
  \end{equation}
  Trivially, in both cases, we map the 0-point arch system to the empty permutation and the 1-point arch system to the singleton permutation 1.

  Now, suppose $\alpha$ is the concatenation $\alpha_1\alpha_2$ of two nonempty arch systems. Then
  $\Pi_{213}(\alpha)$
  is the skew sum $\Pi_{213}(\alpha_1)\ominus\Pi_{213}(\alpha_2)$,
  a copy of $\Pi_{213}(\alpha_1)$ being positioned to the upper left of $\Pi_{213}(\alpha_2)$.
  $\Pi_{132}(\alpha)$ is similar. Otherwise, $\Lambda(\Pi_{213}(\alpha))$ and $\Lambda(\Pi_{132}(\alpha))$ would have an arc connecting some point of $\alpha_1$ to some point of~$\alpha_2$.

\begin{figure}[ht]
$$
\Pi_{213}\big(
  \!
  \raisebox{-6pt}{
  \begin{tikzpicture}[scale=0.15,line join=round]
    \blueArc{5}{2 and 1.5}
    \blueArc{9}{2 and 1.5}
    \blueArc{13}{2 and 1.5}
    \node[] at (3.0,1) {$\alpha$};
    \node[] at (7.0,1) {$\beta$};
    \node[] at (11.0,1) {$\gamma$};
    \plotpermnobox[blue!50!black]{}{1,0,0,0,1,0,0,0,1,0,0,0,1}
  \end{tikzpicture}
  }
  \!
  \big)
  \;\;=\;\;\;
  \raisebox{-63pt}{
  \begin{tikzpicture}[scale=0.35,line join=round]
    \setplotptradius{0.22}
    \fill[blue!5] (.5,1.25) rectangle (13.5,15.2-.96);
    \draw (.5,1.25) rectangle (13.5,15.2-.96);
    \plotpermnobox[blue!50!black]{}{1.75,0,0,0,2.5,0,0,0,3.25,0,0,0,4}
    \fill[blue!10] (1.3+.16,11.3-.64) rectangle (4.7-.16,14.7-.96);
    \draw (1.3+.16,11.3-.64) rectangle (4.7-.16,14.7-.96);
    \node[] at (3,13-.8) {\scriptsize $\Pi_{213}(\alpha)$};
    \fill[blue!10] (5.3+.16,7.8-.32) rectangle (8.7-.16,11.2-.64);
    \draw (5.3+.16,7.8-.32) rectangle (8.7-.16,11.2-.64);
    \node[] at (7,9.5-.48) {\scriptsize $\Pi_{213}(\beta)$};
    \fill[blue!10] (9.3+.16,4.3) rectangle (12.7-.16,7.7-.32);
    \draw (9.3+.16,4.3) rectangle (12.7-.16,7.7-.32);
    \node[] at (11,6-.16) {\scriptsize $\Pi_{213}(\gamma)$};
  \end{tikzpicture}
  }
$$
\caption{Mapping an arch system to a 213-avoider}\label{figArchSysPi}
\end{figure}
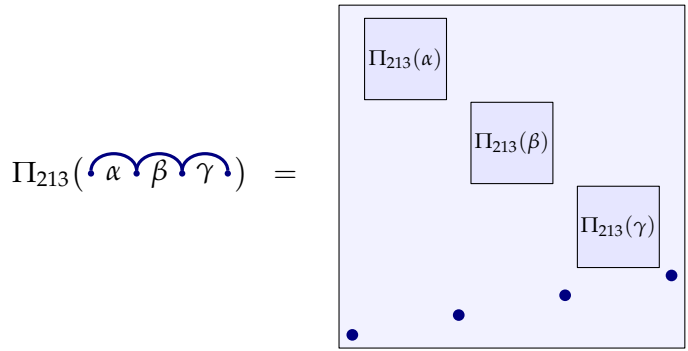

  Finally, suppose $\alpha$ is a sequence of $k$ (possibly empty) arch systems, $\alpha_1,\ldots,\alpha_k$, enclosed in $k$ connected arcs, like
  $
  \!
  \raisebox{-6pt}{
  \begin{tikzpicture}[scale=0.15,line join=round]
    \blueArc{5}{2 and 1.5}
    \blueArc{11}{3 and 1.5}
    \fill[white] (6.3,1.5) rectangle (9.7,3);
    \blueArc{15}{2 and 1.5}
    \node[] at (3.1,1) {$\alpha_1$};
    \node[] at (8,1) {$\ldots$};
    \node[] at (13.1,1) {$\alpha_k$};
    \plotpermnobox[blue!50!black]{}{1,0,0,0,1,0,0,0,0,0,1,0,0,0,1}
  \end{tikzpicture}
  }
  \!
  $.
  Then $\Pi_{213}(\alpha)$ consists of $\Pi_{213}(\alpha_1\ldots\alpha_k)$ \emph{above} the increasing permutation $12\ldots(k+1)$, where
  $\Pi_{213}(\alpha_i)$ is between $i$ and $i+1$ for each $i$. See Figure~\ref{figArchSysPi} for an illustration.
  To satisfy~\eqref{eqLambdaPi}, the endpoints of the arcs must map to consecutive increasing values in the permutation, and each $\Pi_{213}(\alpha_i)$ must be above $\Pi_{213}(\alpha_{i+1})$. To avoid creating an occurrence of 213, each nonempty $\Pi_{213}(\alpha_i)$ must be above $i$ and $i+1$.
  Analogously, to avoid creating a 132, $\Pi_{132}(\alpha)$ consists of $\Pi_{132}(\alpha_1\ldots\alpha_k)$ \emph{below} an increasing permutation of length $k+1$.
\end{proof}

As an aside, we note that the proof of
Proposition~\ref{propArchSystems} can easily be adapted to establish that in $\av_n(213)$ and $\av_n(132)$ each permutation is uniquely determined by
the set consisting of the pairs of values comprising its ascents.

%
%
%
%
\subsection{Arch configurations}

A domino is comprised of a 213-avoiding top cell and a 132-avoiding bottom cell.
Thus, by Proposition~\ref{propArchSystems}, corresponding to each domino is an \emph{arch configuration} consisting of an interleaved pair of arch systems.
See Figure~\ref{figArchConfigEx} for an illustration.
In the figures, the arch system for the top cell is shown above the line, and that for the bottom cell is below the line.
Isolated points are marked with a short strut to indicate to which arch system they belong.

\begin{figure}[ht]
$$
  \begin{tikzpicture}[scale=0.25,line join=round]
    \draw [gray,thin] (0,1)--(36,1);
    \redArc{1}{.5}
    \redArc{4}{12.5 and 8.5}
    \redPt{5}
    \blueArc{7}{.5}
    \blueArc{8}{.5}
    \redArc{9}{3.5}
    \redArc{10}{1}
    \blueArc{11}{1.5}
    \redArc{12}{.5}
    \redArc{13}{1}
    \blueArc{14}{5.5 and 4.5}
    \redArc{16}{6 and 5.5}
    \redArc{19}{4}
    \blueArc{20}{1}
    \blueArc{21}{2}
    \bluePt{22}
    \redArc{23}{1}
    \blueArc{24}{1.5}
    \redPt{26}
    \redArc{29}{3}
    \redArc{31}{.5}
    \blueArc{33}{1.5}
    \bluePt{34}
    \plotpermnobox[blue!50!black] {}{0,0,1,0,0,1,1,1,0,0,1,0,0,1,0,0,1,1,0,1,1,1,0,1,0,0,0,0,0,1,0,0,1,1,0}
    \plotpermnobox[red!50!black]{}{1,1,0,1,1,0,0,0,1,1,0,1,1,0,1,1,0,0,1,0,0,0,1,0,1,1,1,1,1,0,1,1,0,0,1}
  \end{tikzpicture}
$$
\caption{The arch configuration for a domino}
\label{figArchConfigEx}
\end{figure}
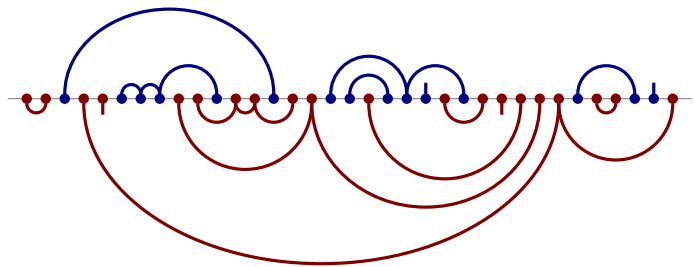

Recall that the only restriction on the cells in a domino is that the following arrangement of points (forming a 1324) must be avoided.
\[
  \begin{tikzpicture}[scale=.8]
    \fill[blue!5] (0,1) rectangle (1,2);
    \fill[red!10] (0,0) rectangle (1,1);
   	\draw[thick, gray] (0,1) -- (1,1);
   	\draw[very thick] (0,0) rectangle (1,2);
    \fill[radius=0.08,red!50!black] (.2,.33) circle;
    \fill[radius=0.08,blue!50!black] (.36,1.33) circle;
    \fill[radius=0.08,red!50!black] (.64,.67) circle;
    \fill[radius=0.08,blue!50!black] (.8,1.67) circle;
 \end{tikzpicture}
\]
The arch configuration corresponding to this is $\archavoid$.
Indeed, avoiding this pattern of arcs in an arch configuration is equivalent to avoiding 1324 in a domino.
\begin{prop}\label{propArchConfigs}
  The set $\D$ of dominoes is in bijection with arch configurations that do not contain the pattern $\archavoid$.
\end{prop}
\begin{proof}
  By the bijection used in the proof of Proposition~\ref{propArchSystems},
  an arch configuration contains an occurrence of $\archavoid$
  if and only if the
  corresponding pair of cells contains an occurrence of 1324 of the form \mbox{$k\ell(k+1)(\ell+1)$}, for values $k$ and $\ell$ such that $\ell>k+1$.
  So, if an arch configuration contains $\archavoid$, the corresponding gridded permutation contains 1324.

  For the converse, it suffices to show that
  if a permutation gridded in $\begin{gridmx}\av(213)\\\av(132)\end{gridmx}$ contains
  an occurrence of 1324, then it contains some, possibly distinct, occurrence of 1324 that has the form \mbox{$k\ell(k+1)(\ell+1)$}.
  Suppose $acbd$ is an occurrence of 1324, gridded in $\begin{gridmx}\av(213)\\\av(132)\end{gridmx}$, where $a<b<c<d$.
  Then $a$ and $b$ are in the bottom, 132-avoiding, cell.
  Consider the set of values in the interval \mbox{$I=\{a,a+1,\ldots,b-1\}$}.
  These must all occur to the left of $b$, otherwise a 132 would be formed.
  Let $a+i$, where $i\geqs0$, be the greatest element of $I$ that occurs to the left of $c$; this value must exist since $a$ itself occurs before~$c$.
  Then $(a+i)c(a+i+1)d$ is an occurrence of 1324 in which the first and third values differ by one.

  Applying an analogous argument to the interval \mbox{$J=\{c+1,\ldots,d-1,d\}$} then
  yields $j\geqs0$ such that $(a+i)(d-j-1)(a+i+1)(d-j)$ is an occurrence of 1324 with the required form.
\end{proof}

To enumerate dominoes, we construct a functional equation for arch configurations, which we then solve.
We build arch configurations from left to right.
A vertical line positioned between two points of an arch configuration may intersect some arcs.
We call the partial arch configuration to the left of such a line an \emph{arch prefix}; any arcs intersected by the line are \emph{open}.

Let $\A$ be the set of arch prefixes with no open \emph{upper} arcs, and
let $A(v)=A(z,v)$ be the ordinary generating function for $\A$,
in which
$z$ marks points and
$v$ marks open \emph{lower} arcs. Thus, $A(0)=A(z,0)$ is the generating function for the set of dominoes~$\D$.

\begin{prop}\label{propFuncEqArches}
The generating function $A(v)=A(z,v)$, for the set $\A$ of arch prefixes with no open upper arcs, in which $z$ marks points and
$v$ marks open lower arcs, satisfies the functional equation
\begin{equation}
	\label{eq-func-eq-leaves-no-leaves}
	A(v)
	\;=\;
	\frac{1}{1-z\+A(v)}\:+\:
	z\+(1+v)\+\left(A(v)\:+\:\frac{A(v)-A(0)}{v}\right).
\end{equation}
\end{prop}

\begin{proof}
There are six possible ways in which a non-empty element of $\A$
can be
decomposed, depending on its rightmost point. These are illustrated in Figure~\ref{figArchPrefixes}.

\begin{figure}[ht]
$$
  (i)\begin{tikzpicture}[scale=0.24,line join=round]
    \fillBlueArc{6}{2.5}
    \draw [gray,thin] (0,1)--(8,1);
    \draw [gray!60!black,thick] (6.5,4.5)--(6.5,-5);
    \node[] at (3.5,2.25) {$\A$};
    \begin{scope}
      \clip (0,-6) rectangle (7.75,6);
      \redArc{3}{4}
      \redArc{4}{2.5}
    \end{scope}
    \redPt{7}
    \plotpermnobox[red!50!black]{}{0,0,1,1,0,0,1}
  \end{tikzpicture}
  \qquad\quad
  (ii)\begin{tikzpicture}[scale=0.24,line join=round]
    \fillBlueArc{6}{2.5}
    \draw [gray,thin] (0,1)--(8,1);
    \draw [gray!60!black,thick] (6.5,4.5)--(6.5,-5);
    \node[] at (3.5,2.25) {$\A$};
    \begin{scope}
      \clip (0,-6) rectangle (7.75,6);
      \redArc{3}{4}
      \redArc{4}{2.5}
      \redArc{7}{1}
    \end{scope}
    \plotpermnobox[red!50!black]{}{0,0,1,1,0,0,1}
  \end{tikzpicture}
  \qquad\quad
  (iii)\begin{tikzpicture}[scale=0.24,line join=round]
    \fillBlueArc{6}{2.5}
    \draw [gray,thin] (0,1)--(8,1);
    \draw [gray!60!black,thick] (6.5,4.5)--(6.5,-5);
    \node[] at (3.5,2.25) {$\A$};
    \begin{scope}
      \clip (0,-6) rectangle (7.75,6);
      \redArc{3}{4}
      \redArc{4}{1.5}
    \end{scope}
    \plotpermnobox[red!50!black]{}{0,0,1,1,0,0,1}
  \end{tikzpicture}
  \qquad\quad
  (iv)\begin{tikzpicture}[scale=0.24,line join=round]
    \fillBlueArc{6}{2.5}
    \draw [gray,thin] (0,1)--(8,1);
    \draw [gray!60!black,thick] (6.5,4.5)--(6.5,-5);
    \node[] at (3.5,2.25) {$\A$};
    \begin{scope}
      \clip (0,-6) rectangle (7.75,6);
      \redArc{3}{4}
      \redArc{4}{1.5}
      \redArc{7}{1.5}
    \end{scope}
    \plotpermnobox[red!50!black]{}{0,0,1,1,0,0,1}
  \end{tikzpicture}
  $$

  $$
  (v)\begin{tikzpicture}[scale=0.24,line join=round]
    \draw [white] (1,4.5)--(1,-10); 
    \fillBlueArc{6}{2.5}
    \draw [gray,thin] (0,1)--(8,1);
    \draw [gray!60!black,thick] (6.5,4.5)--(6.5,-5);
    \node[] at (3.5,2.25) {$\A$};
    \begin{scope}
      \clip (0,-6) rectangle (7.75,6);
      \redArc{3}{4}
      \redArc{4}{2.5}
    \end{scope}
    \bluePt{7}
    \plotpermnobox[red!50!black]{}{0,0,1,1,0,0,0}
    \plotpermnobox[blue!50!black] {}{0,0,0,0,0,0,1}
  \end{tikzpicture}
  \qquad\quad
  (vi)\begin{tikzpicture}[scale=0.24,line join=round]
    \draw [white] (1,4.5)--(1,-10); 
    \fillBlueArc{6}{2.5}
    \fillBlueArc{12}{2}
    \fillBlueArc{17}{1.5}
    \fillBlueArc{25}{3}
    \draw [gray,thin] (0,1)--(27,1);
    \draw [gray!60!black,thick] (25.5,6)--(25.5,-9.5);
    \blueArc{13}{3}
    \blueArc{18}{2.5}
    \blueArc{26}{4}
    \node[] at (3.35,2.25) {$\A$};
    \node[] at (9.85,2) {$\A$};
    \node[] at (15.35,1.825) {$\A$};
    \node[] at (21.85,2.25) {$\A$};
    \begin{scope}
      \clip (0,-10) rectangle (26.75,6);
      \redArc{3}{19.5 and 9}
      \redArc{5}{18 and 8}
      \redArc{11}{13 and 6}
      \redArc{19}{6.5 and 4.5}
      \redArc{23}{3.5 and 2.5}
      \redArc{24}{2.5 and 1.5}
    \end{scope}
    \plotpermnobox[blue!50!black] {}{0,0,0,0,0,0,1,0,0,0,0,0,1,0,0,0,0,1,0,0,0,0,0,0,0,1}
    \plotpermnobox[red!50!black]{}{0,0,1,0,1,0,0,0,0,0,1,0,0,0,0,0,0,0,1,0,0,0,1,1,0,0}
  \end{tikzpicture}
  $$
\caption{The six ways of decomposing a non-empty arch prefix in $\A$}
\label{figArchPrefixes}
\end{figure}
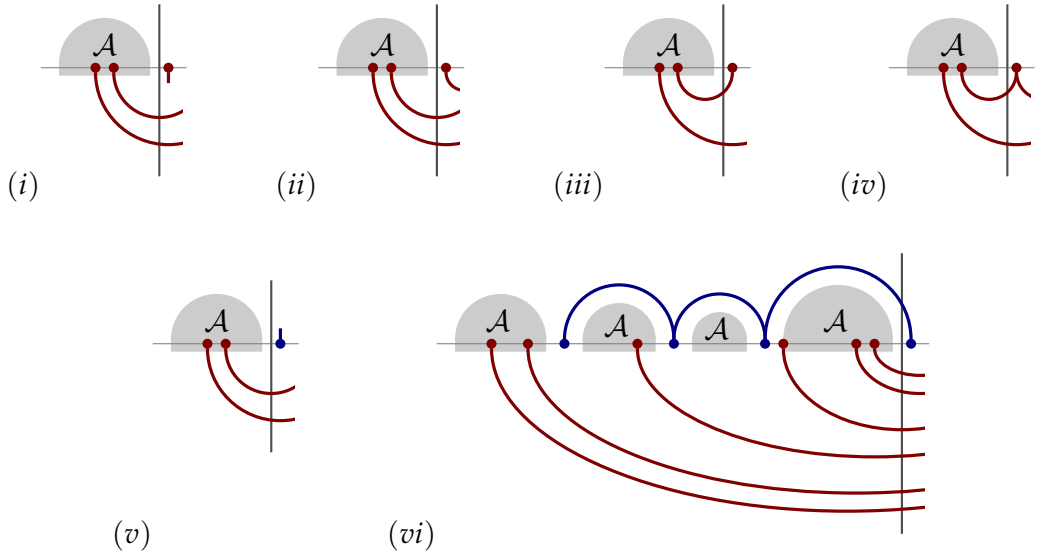

If the rightmost point belongs to the lower arch system,
then there are four cases: (\emph{i})~an isolated point, (\emph{ii})~the left endpoint of an arc, (\emph{iii})~the right endpoint of an arc, and (\emph{iv})~both the left and right endpoint of an arc.
These contribute the following terms to the functional equation for $A(v)$:
$$
(i)~~z\+A(v) \qquad (ii)~~z\+v\+A(v) \qquad (iii)~~z\+v^{-1}\big(A(v)-A(0)\big) \qquad (iv)~~z\big(A(v)-A(0)\big)  .
$$

If the rightmost point belongs to the upper arch system,
then, since there are no open upper arcs, it is either (\emph{v})~an isolated point, or else (\emph{vi})~the right endpoint of an arc.
In the former case, this contributes $z\+A(v)$ to the functional equation for $A(v)$.
In the latter case, the arch prefix decomposes into a connected sequence of one or more upper arcs, each enclosing an element of $\A$ (possibly empty), preceded by a further initial element of $\A$ (also possibly empty). This makes a contribution of
$$
\frac{z^2\+A(v)^2}{1-z\+A(v)}
$$
to the functional equation for $A(v)$.

Summing these terms, including a term for the empty prefix, and simplifying, yields the functional equation in the statement of the proposition.
\end{proof}

%
%
%
%
\subsection{The enumeration of dominoes}

To complete the proof of Theorem~\ref{thm-twocell}, we employ resultant methods to eliminate the variables $v$ and $A(v)$ from the functional equation~\eqref{eq-func-eq-leaves-no-leaves}. This yields a minimal polynomial for $A(0)$ which we then use to derive the closed-form formula for the number of dominoes and their exponential growth rate.

\begin{proof}[Proof of Theorem~\ref{thm-twocell}]
Clearing denominators from~\eqref{eq-func-eq-leaves-no-leaves} and moving all terms to one side yields
\[
	0 \;=\; P(A(v), A(0), z, v)
\]
where $P$ is the polynomial
\[
	P(x, y, z, v) \;=\; (zv - z^2(1+v)^2)x^2 + z^2(1+v)xy + (z(1+v)^2-v)x - z(1+v)y + v.
\]
The presence of the term $x^2$ indicates that the kernel method does not apply here. Instead, we use a more general method of Bousquet-M\'{e}lou and Jehanne~\cite{bousquet-melou:poly-eqs} which says that $A(v)$ and $v$ can be eliminated from the functional equation via \emph{iterated discriminants}. Specifically, define
\[
	Q(y, z) \;=\; \discrim_v\!\big(\discrim_{x}\!\big(P(x, y, z, v)\big)\big).
\]
Then it follows that the minimal polynomial for $A(0)$ is one of the irreducible factors of $Q(y, z)$. Performing the calculation, we find that
\[
	Q(y,z) \;=\; -256z^8R_1(y, z)^2R_2(y,z),
\]
where
\begin{align*}
	R_1(y,z) &\;=\; z^3y^2 + z(1-4z)y + 4z - 1 ,\\
	R_2(y,z) &\;=\; z^4y^3 + 2z^2(3z+1)y^2 + (12z^2 - 10z + 1)y + 8z - 1 .
\end{align*}
The two series solutions of $0 = R_1(y,z)$ begin $y = z^{-1} + O(1)$ and $y = -z^{-2} + O(z^{-1})$, which do not match the known initial terms of $A(0)$.
Therefore, it is $R_2$ that is a minimal polynomial for $A(0)$.

We verify that, for each $n$, the coefficient of $z^n$ in the series expansion of $A(0)$ is given by
\[
\frac{2(3n+3)!}{(n + 2)!(2n + 3)!},
\]
by using \emph{Mathematica}~\cite{Mathematica}.

\begin{notebox}{}
\begin{footnotesize}
\begin{verbatim}
minpoly[y_] := z^4y^3 + 2z^2(3z + 1)y^2 + (12z^2 - 10z + 1)y + 8z - 1
series = Sum[2(3n + 3)!/((n + 2)!(2n + 3)!) z^n, {n, 0, Infinity}]
\end{verbatim}
\vspace{-14pt}
\begin{verbatim}
   (2(-1 - 3z + Hypergeometric2F1[-2/3, -1/3, 1/2, 27z/4]))/(3z^2)

minpoly[series] // FunctionExpand // Simplify
\end{verbatim}
\vspace{-14pt}
\begin{verbatim}
   0
\end{verbatim}
\end{footnotesize}
\end{notebox}

The first command assigns the known minimal polynomial for $A(0)$ to the variable \texttt{minpoly}.
The second command creates the power series that we want to verify is equal to $A(0)$; \emph{Mathematica} deduces a nice form for this.
The final command substitutes the power series into the minimal polynomial and simplifies. The result is 0, so the power series satisfies the minimal polynomial.
Since the initial terms of the power series coincide with those of $A(0)$ and not with those of the other roots of $R_2$, this completes the proof of the first part of Theorem~\ref{thm-twocell}.

To derive the growth rate, note that the exponential growth rate of an algebraic generating function (and, in fact, a complete asymptotic expansion) can be derived from the minimal polynomial using the method outlined by Flajolet and Sedgewick~\cite[Note VII.36]{flajolet:ac}. The exponential growth rate must be the reciprocal of one of the roots of the discriminant of the minimal polynomial with respect to $y$. Since
\[
	\discrim_{y}\!\big(z^4y^3 + 2z^2(3z+1)y^2 + (12z^2 - 10z + 1)y + 8z - 1\big) \;=\; -z^5(27z-4)^3,
\]
and with the knowledge that algebraic generating functions for combinatorial sequences are analytic at the origin~\cite[Proposition 3.1]{klazar:bell-numbers}, we conclude that the exponential growth rate for the power series of $A(0)$ is $27/4 = 6.75$.
\end{proof}

The counting sequence for dominoes is \href{http://oeis.org/A000139}{A000139} in \emph{OEIS}~\cite{OEIS}.
Among other things, this enumerates West-two-stack-sortable permutations~\cite{zeilberger:west}, rooted nonseparable planar maps~\cite{brown:non-sep-planar-maps}
and a class of branching polyominoes known as \emph{fighting fish}~\cite{DGRS2017,DGRS2016,Fang2017}.
So far, we have not been able to establish a bijection between dominoes and any of these structures.
\begin{prob}
Find a bijection between 1324-avoiding dominoes and another combinatorial class known to be equinumerous.
\end{prob}

%
%
%
%
\subsection{Balanced dominoes}\label{subsec-balanced}

We say that a domino is \emph{balanced} if its top cell contains the same number of points as its bottom cell.
Let
$\B$ be the set of balanced dominoes and
$\B_m$ be the set of balanced dominoes having a total of $2m$ points, $m$ points in each cell.
We define the growth rate of balanced dominoes to be $\gr(\B) = \lim_{m\to\infty}\sqrt[2m]{|\B_m|}$.
We prove that the growth rate of \emph{balanced} dominoes is the same as that of \emph{all} dominoes.
This result is used in Sections~\ref{secLowerBound} and~\ref{secLowerBound2} where our lower bound
constructions
consist of balanced dominoes.

\begin{prop}\label{propBalancedDominoes}
    The growth rate of balanced dominoes 
    is 27/4.
\end{prop}

In the proof, we use two elementary manipulations of dominoes.
Given a domino $\sigma$, let the $180^\circ$~rotation of $\sigma$ be denoted $\overset{\curvearrowleft}{\sigma}$. This is itself a valid domino.
Also, given two dominoes $\sigma$ and~$\tau$, define $\sigma \varoast \tau$ to be the domino whose arch configuration is produced by concatenating the arch configurations of $\sigma$ and $\tau$.

\begin{proof}[Proof of Proposition~\ref{propBalancedDominoes}]
  Let $d(t,b)$ denote the number of $(t+b)$-point dominoes with $t$ points in the top cell and $b$ points in the bottom cell.
  For a given $m$, let $t_m$ be a value of $t$ that maximises $d(t,m-t)$. Let $d_{\max}=d(t_m,m-t_m)$ be this maximal value.
  Since $0\leqs t \leqs m$, there are only $m+1$ possible choices for $t_m$. Hence by the pigeonhole principle,
	\[
		d_{\max} \;\geqs\; \frac{|\D_{m}|}{m+1}.
	\]
Let $\sigma$ and $\tau$ be any two $m$-point dominoes with $t_m$ points in the top cell and $m-t_m$ points in the bottom cell.
Consider the domino $\rho=\sigma \varoast \overset{\curvearrowleft}{\tau}$, whose arch configuration is constructed by concatenating the arch configuration of $\sigma$ and the arch configuration of the $180^\circ$ rotation of $\tau$.
This is a balanced domino in $\B_m$.
Moreover, $\sigma$ and $\tau$ can be recovered from $\rho$ simply by splitting its arch configuration into two halves. Thus,
\[
    |\B_{m}| \;\geqs\; d_{\max}^2 \;\geqs\; \frac{|\D_m|^{2}}{(m+1)^{2}}.
\]
Since it is also the case that $|\D_{2m}|\geqs|\B_m|$, it follows, by taking the $2m$th root, and the limit as $m$ tends to infinity, that $\gr(\B) = \gr(\D) = 27/4$.
\end{proof}

%
%
%
%
%
%
%
%
%
%
%
%
\section{An upper bound}\label{secUpperBound}

In this section, we use the results of Section~\ref{sec-dominoes} to establish a new upper bound on the growth rate of the 1324-avoiders.
Our upper bound follows from the fact that we can split a 1324-avoider, gridded in the staircase, in such a way as to produce a domino.

\begin{theorem}\label{thmUpperBound}The growth rate of
$\av(1324)$ is at most $27/2=13.5$.
\end{theorem}

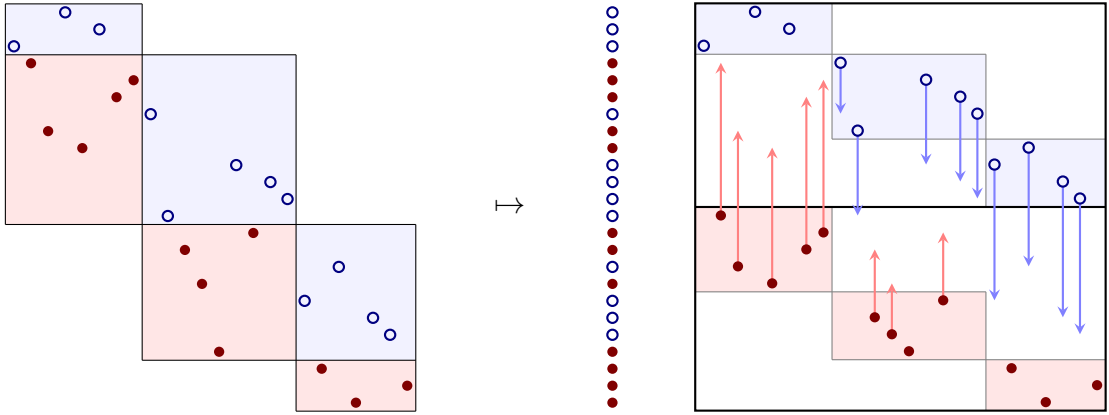
\begin{figure}[ht]
  \begin{center}
    \setplotptradius{0.3}
    \begin{tikzpicture}[scale=0.225]
      \fill[blue!5] (3.5,26.5) rectangle (11.5,23.5);
      \fill[red!10] (3.5,23.5) rectangle (11.5,13.5);
      \fill[blue!5] (11.5,23.5) rectangle (20.5,13.5);
      \fill[red!10] (11.5,13.5) rectangle (20.5,5.5);
      \fill[blue!5] (20.5,13.5) rectangle (27.5,5.5);
      \fill[red!10] (20.5,5.5) rectangle (27.5,2.5);
      \setplotptfun{\draw}
      \plotpermnobox[blue!50!black,thick]{}{ 0,  0,  0, 24,  0,  0, 26,  0, 25,  0,  0, 20, 14,  0,  0, 0, 17,  0, 16, 15, 9, 0, 11, 0, 8, 7, 0}
      \setplotptfun{\fill}
      \plotpermnobox[ red!50!black]{}{ 0,  0,  0,  0, 23, 19,  0, 18,  0, 21, 22,  0,  0, 12, 10, 6,  0, 13,  0,  0, 0, 5,  0, 3, 0, 0, 4}
      \draw[thin] (3.5,26.5) -- (11.5,26.5);
      \draw[thin] (3.5,13.5) -- (3.5,26.5);
      \draw[thin] (3.5,23.5) -- (20.5,23.5);
      \draw[thin] (11.5,5.5) -- (11.5,26.5);
      \draw[thin] (3.5,13.5) -- (27.5,13.5);
      \draw[thin] (20.5,2.5) -- (20.5,23.5);
      \draw[thin] (11.5,5.5) -- (27.5,5.5);
      \draw[thin] (27.5,2.5) -- (27.5,13.5);
      \draw[thin] (20.5,2.5) -- (27.5,2.5);
      \node at (33,14.5) {$\mapsto$};
      \setplotptfun{\draw}
      \plotvert[blue!50!black,thick]{39}{ 0,  0,  0, 24,  0,  0, 26,  0, 25,  0,  0, 20, 14,  0,  0, 0, 17,  0, 16, 15, 9, 0, 11, 0, 8, 7, 0}
      \setplotptfun{\fill}
      \plotvert[ red!50!black]{39}{ 0,  0,  0,  0, 23, 19,  0, 18,  0, 21, 22,  0,  0, 12, 10, 6,  0, 13,  0,  0, 0, 5,  0, 3, 0, 0, 4}
    \end{tikzpicture}
    $\qquad$
    \begin{tikzpicture}[scale=0.225,>=stealth]
      \fill[blue!5] (0.5,24.5) rectangle (8.5,21.5);
      \fill[blue!5] (8.5,21.5) rectangle (17.5,16.5);
      \fill[blue!5] (17.5,16.5) rectangle (24.5,12.5);
      \fill[red!10] (0.5,12.5) rectangle (8.5,7.5);
      \fill[red!10] (8.5,7.5) rectangle (17.5,3.5);
      \fill[red!10] (17.5,3.5) rectangle (24.5,0.5);
      \draw[thin,gray] (.5,21.5)--(17.5,21.5);
      \draw[thin,gray] (8.5,16.5)--(24.5,16.5);
      \draw[thin,gray] (.5,7.5)--(17.5,7.5);
      \draw[thin,gray] (8.5,3.5)--(24.5,3.5);
      \draw[thin,gray] (8.5,3.5)--(8.5,12.5);
      \draw[thin,gray] (8.5,16.5)--(8.5,24.5);
      \draw[thin,gray] (17.5,0.5)--(17.5,7.5);
      \draw[thin,gray] (17.5,12.5)--(17.5,21.5);
      \draw[thick] (.5,.5) rectangle (24.5,24.5);
      \draw[thick] (.5,12.5)--(24.5,12.5);
      \draw[blue!50!white,thick,->] ( 9,21-0.25)->( 9,18);
      \draw[blue!50!white,thick,->] (14,20-0.25)->(14,15);
      \draw[blue!50!white,thick,->] (16,19-0.25)->(16,14);
      \draw[blue!50!white,thick,->] (17,18-0.25)->(17,13);
      \draw[blue!50!white,thick,->] (10,17-0.25)->(10,12);
      \draw[blue!50!white,thick,->] (20,16-0.25)->(20, 9);
      \draw[blue!50!white,thick,->] (18,15-0.25)->(18, 7);
      \draw[blue!50!white,thick,->] (22,14-0.25)->(22, 6);
      \draw[blue!50!white,thick,->] (23,13-0.25)->(23, 5);
      \draw[red!50,thick,->] ( 2,12)->( 2,21);
      \draw[red!50,thick,->] ( 8,11)->( 8,20);
      \draw[red!50,thick,->] ( 7,10)->( 7,19);
      \draw[red!50,thick,->] ( 3, 9)->( 3,17);
      \draw[red!50,thick,->] ( 5, 8)->( 5,16);
      \draw[red!50,thick,->] (15, 7)->(15,11);
      \draw[red!50,thick,->] (11, 6)->(11,10);
      \draw[red!50,thick,->] (12, 5)->(12, 8);
      \setplotptfun{\draw}
      \plotpermnobox[blue!50!black,thick]{}{22,  0, 0, 24, 0, 23,  0,  0, 21, 17, 0, 0, 0, 20, 0, 19, 18, 15, 0, 16, 0, 14, 13, 0}
      \setplotptfun{\fill}
      \plotpermnobox[ red!50!black]{}{ 0, 12, 9,  0, 8,  0, 10, 11,  0,  0, 6, 5, 4,  0, 7,  0,  0,  0, 3,  0, 1,  0,  0, 2}
    \end{tikzpicture}
  \end{center}
\caption{Mapping a greedy-gridded 1324-avoider to a binary word and a domino}
\label{figUpperBound}
\end{figure}

\begin{proof}
  We define an injection from $\av_n(1324)$ into the Cartesian product $\{{\color{blue!50!black}\boldsymbol\circ},{\color{red!50!black}\bullet}\}^n \times \D_n$, for every~$n\geqs1$, each permutation being mapped to a pair consisting of a binary word (over the alphabet $\{{\color{blue!50!black}\boldsymbol\circ},{\color{red!50!black}\bullet}\}$) and a domino.
  See Figure~\ref{figUpperBound} for an illustration.
  Given a 1324-avoider $\sigma$, let $\sigma^\#$ be the greedy gridding of $\sigma$ in the descending $\big(\!\av(213),\av(132)\big)$ staircase.

  The binary word is constructed by reading the points of $\sigma$ from top to bottom and recording a ring ($\color{blue!50!black}\boldsymbol\circ$) if the point is in an upper, $\av(213)$, cell of $\sigma^\#$, and recording a disk ($\color{red!50!black}\bullet$) if it is in a lower, $\av(132)$, cell.

  The domino is constructed by placing all the points from the upper cells of $\sigma^\#$ in the top cell of the domino, retaining their horizontal positions, and similarly placing the points from the lower cells of $\sigma^\#$ in the bottom cell of the domino.
  The result is a valid domino
  since the points gridded in the upper cells of $\sigma^\#$ collectively avoid 213, the points gridded in the lower cells collectively avoid 132 and no additional occurrence of 1324 can be created by splitting $\sigma^\#$ in this way.

  This mapping is an injection,
  because the original permutation $\sigma$ can be recovered from the domino by repositioning the points vertically according to the
  information in the binary word, as illustrated by the arrows in Figure~\ref{figUpperBound}.

  There are $2^n$ binary words of length $n$ and, by Theorem~\ref{thm-twocell}, the growth rate of the set of dominoes $\D$ is $27/4$.
  Therefore, the
  union of the
  Cartesian products of binary words and dominos of each size, $\bigcup_{n\geqs1}\big(\{{\color{blue!50!black}\boldsymbol\circ},{\color{red!50!black}\bullet}\}^n \times \D_n\big)$,
  has growth rate $2 \times27/4=13.5$. The existence of the injection establishes that this value is an upper bound on the growth rate of $\av(1324)$.
\end{proof}

The use of an arbitrary binary word to record the vertical interleaving of the points is very rudimentary.
One would hope that the approach could be refined by recording this information as decorations on the domino in such a way as to
yield a tighter upper bound, but we have not been able to do so.

%
%
%
%
\section{An initial lower bound}\label{secLowerBound}

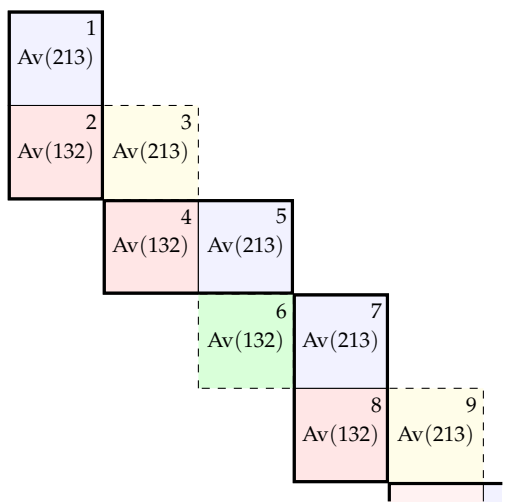
\begin{figure}[ht]
\begin{center}
\begin{tikzpicture}[scale=1.25]
\foreach \x [evaluate=\x as \xx using int(2*\x)] in {1,2,4} {
  \draw[fill=red!10] (\x,-\x) rectangle (\x+1,1-\x);
  {\node at (\x+0.5, 0.5-\x) {\scriptsize $\av(132)$};}
  \node at (\x+0.875,0.825-\x) {\scriptsize \xx};
}
\fill[red!5] (5,-4) rectangle (6,-4.2);
\foreach \x [evaluate=\x as \xx using int(2*\x-1)] in {1,3,4} {
  \draw[fill=blue!5] (\x,1-\x) rectangle (\x+1,2-\x);
  {\node at (\x+0.5, 1.5-\x) {\scriptsize $\av(213)$};}
  \node at (\x+0.875, 1.825-\x) {\scriptsize \xx};
}
\fill[blue!5] (6,-4) rectangle (6.2,-4.2);
\foreach \x [evaluate=\x as \xx using int(2*\x)] in {3} {
  \draw[dashed,fill=green!15] (\x,-\x) rectangle (\x+1,1-\x);
  {\node at (\x+0.5, 0.5-\x) {\scriptsize $\av(132)$};}
  \node at (\x+0.875, 0.825-\x) {\scriptsize \xx};
}
\foreach \x [evaluate=\x as \xx using int(2*\x-1)] in {2,5} {
  \draw[dashed,fill=yellow!10] (\x,1-\x) rectangle (\x+1,2-\x);
  {\node at (\x+0.5, 1.5-\x) {\scriptsize $\av(213)$};}
  \node at (\x+0.875, 1.825-\x) {\scriptsize \xx};
}
  \draw[very thick] (2-0.01,1.01-2) rectangle (2-0.99,2.99-2);
  \draw[very thick] (0.01+2,0.99-2) rectangle (2+1.99,0.01-2);
  \draw[very thick] (5-0.01,1.01-5) rectangle (5-0.99,2.99-5);
  \draw[          ] (6,-4)--(6,-4.2);
  \draw[very thick] (5.01,-4.2)--(5.01,0.99-5)--(6.2,0.99-5);
\end{tikzpicture}
\end{center}
\caption{The decomposition of the staircase into dominoes and connecting cells}
\label{figLowerBoundStaircase}
\end{figure}

Our lower bounds depend on exploiting a specific
partitioning of the staircase.
We decompose the staircase into an alternating sequence of dominoes and individual \emph{connecting cells}.
See Figure~\ref{figLowerBoundStaircase} for an illustration.
In the figure, dominoes are bordered by thick black lines and
connecting cells have dashed borders.
Specifically, if we number the cells $1,2,\ldots$, descending from the top left, as in the figure, then the decomposition is as follows. For each $j\geqs0$:
\vspace{-9pt}\begin{itemize}\itemsep0pt
  \item Cells numbered $6j+1$ and $6j+2$ form a (vertical) domino.
  \item Cells numbered $6j+3$ are connecting cells avoiding 213.
  \item Cells numbered $6j+4$ and $6j+5$ form a domino reflected about the line $y=x$ (a \emph{horizontal domino}). The left cell avoids 132 and the right cell avoids 213.
  \item Cells numbered $6j+6$ are connecting cells avoiding 132.
\end{itemize}\vspace{-9pt}

Observe that any occurrence of 1324 in the staircase is
contained in a pair of adjacent cells, with two points in each cell.
By definition, dominoes avoid 1324.
So, to avoid 1324 in this decomposition of the staircase, it is only necessary to guarantee that
an occurrence of 1324 is not created from two points in a connecting cell and two points in an adjacent domino cell.

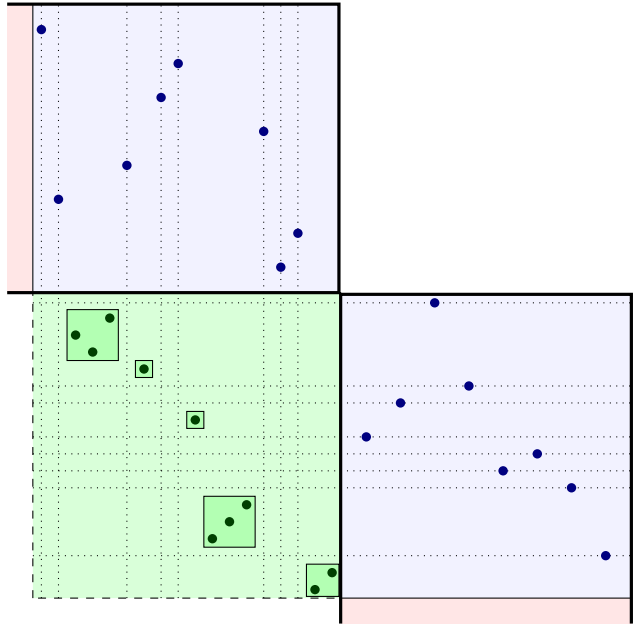
\begin{figure}[ht]
\begin{center}
  \tikzset{brace/.style= {decoration={brace, mirror, amplitude=4pt}, decorate}}
  $\qquad$
  \begin{tikzpicture}[scale=0.225]
    \fill[fill=red!10] (-1,19.5) rectangle (0.5,36.5);
    \fill[blue!5] (0.5,19.5) rectangle (18.5-.1,36.5);
    \fill[blue!5] (18.5,1.5) rectangle (35.5,19.5-.1);
    \fill[fill=red!10] (18.5,0) rectangle (35.5,1.5);
    \draw[dashed,fill=green!15] (0.5,1.5) rectangle (18.5,19.5);
    \draw[fill=green!30] (2.5,15.5) rectangle (5.5,18.5);
    \draw[fill=green!30] (6.5,14.5) rectangle (7.5,15.5);
    \draw[fill=green!30] (9.5,11.5) rectangle (10.5,12.5);
    \draw[fill=green!30] (10.5,4.5) rectangle (13.5,7.5);
    \draw[fill=green!30] (16.5,1.6) rectangle (18.4,3.5);
    \draw[very thick] (-1,19.5)--(18.5-.1,19.5)--(18.5-.1,36.5)--(-1,36.5);
    \draw[] (0.5,19.5)--(0.5,36.5);
    \draw[very thick] (18.5,0)--(18.5,19.5-.1)--(35.5,19.5-.1)--(35.5,0);
    \draw[] (18.5,1.5)--(35.5,1.5);
    \foreach \x in {1,2,6,8,9,14,15,16}
      \draw[dotted] (\x,1.5) -- (\x,36.5);
    \foreach \y in {4,8,9,10,11,13,14,18.9}
      \draw[dotted] (0.5,\y) -- (35.5,\y);
    \plotpermnobox[ blue!50!black]{}{35,25, 0, 0, 0,27, 0,31,33, 0, 0, 0, 0,29,21,23, 0, 0, 0,11, 0,13, 0,18.9, 0,14, 0, 9, 0,10, 0, 8, 0, 4}
    \plotpermnobox[green!25!black]{}{ 0, 0,17,16,18, 0,15, 0, 0,12, 5, 6, 7, 0, 0, 0, 2, 3}
  \end{tikzpicture}
\end{center}
  \caption{Interleaving the skew indecomposable components in a connecting cell with the points in the two adjacent domino cells}
  \label{figInterleaving}
\end{figure}

Recall that every permutation
has a unique representation as the skew sum of a sequence of skew indecomposable components.
For brevity, we will refer to a skew indecomposable component simply as a \emph{component}.
To ensure that there is no occurrence of 1324, it is sufficient to require that every point in a domino cell
is positioned \emph{between} the  components in the adjacent connecting cells.
For example, if a domino cell is to the right of a connecting cell, then this restriction ensures that there is no occurrence of 132 in which the 13 is in the connecting cell and the 2 is in the domino cell.
See Figure~\ref{figInterleaving} for an illustration of a 132-avoiding connecting cell and its adjacent domino cells.

This construction enables us to establish a new lower bound on the growth rate of 1324-avoiders.

\begin{theorem}\label{thmLowerBound1}The growth rate of
$\av(1324)$ is at least $81/8=10.125$.
\end{theorem}

To prove this, we take an approach similar to that used by Bevan in~\cite{BevanAv1324GR}.

\begin{proof}
For each $k\geqs1$, let $\PPP_k$ be the set of gridded permutations, gridded in the first~$3k$ cells of the staircase, decomposed as described above, with every point in a domino cell
positioned between the skew indecomposable components in adjacent connecting cells, satisfying the following three conditions.
\vspace{-9pt}\begin{itemize}\itemsep0pt
  \item Each domino cell contains $14k$ points.
  \item Each connecting cell contains $8k$ points.
  \item The permutation in each connecting cell has $7k$ skew indecomposable components.
\end{itemize}\vspace{-9pt}
(These numbers were chosen by performing the calculations for arbitrary ratios and determining the values that maximise the growth rate.)

Each element of $\PPP_k$ is a gridded $36k^2$-point permutation.
The number of these gridded permutations is exactly
\[
\big|\PPP_k\big|  \;=\;  \big|\B_{14k}\big|^k \, \big|\C_{8k,\+7k}\big|^k \, \binom{21k}{14k}^{\!2k-1} ,
\]
where $\B_n$ is, as before, the set of balanced dominoes with $n$ points in each cell, and $\C_{n,\+c}$ is the set of $n$-point 213-avoiders (or 132-avoiders) with $c$ skew indecomposable components.
The final binomial coefficient counts the number of possible ways of interleaving $14k$ points in a domino cell with $7k$ skew indecomposable components in an adjacent connecting cell.

From Proposition~\ref{propBalancedDominoes}, we know that $|\B_n|=(27/4)^{2n}\cdot\theta(n)$, where $\lim_{n\to\infty} \sqrt[n]{\theta(n)}=1$.
It is also known that
$|\C_{n,\+c}|=\frac{c}{n}\+\binom{2n-c-1}{n-1}$, since $\C_{n,\+c}$ is equinumerous
with the number of $n$-vertex Catalan forests with $c$ trees
(see~\cite{flajolet:ac}~Example~III.8).

Thus, using Stirling's approximation to determine the asymptotics of the binomial coefficients,
\begin{align*}
\lim_{k\to\infty}
\big|\PPP_k\big|^{1/36k^2}
&
\;=\;
\lim_{k\to\infty}
\left[
\left(\frac{27}{4}\right)^{\!28k^2}
\theta(14k)^k
\cdot
\left(\frac{7}{8}\right)^{\!k}
\binom{9k-1}{8k-1}^{\!k}
\cdot
\binom{21k}{14k}^{\!2k-1}
\right]^{\!1/36k^2}
\\[7.5pt]
&
\;=\;
\frac{3^{\+7/3}}{4^{\+7/9}}
\,\cdot\,
\frac{3^{\+1/2}}{2^{\+2/3}}
\,\cdot\,
\frac{3^{\+7/6}}{2^{\+7/9}}
\;\;=\;\;
\frac{81}{8}
.
\end{align*}

An $n$-point permutation can be gridded in $j$ cells in at most \[\binom{n+\ceil{(j-1)/2}}{\ceil{(j-1)/2}}\binom{n+\floor{(j-1)/2}}{\floor{(j-1)/2}}\] ways (the number of ways of choosing the positions of the $j-1$ horizontal and vertical cell dividers without restriction).
So the number of ways of gridding a $36k^2$-point permutation in $3k$ cells is no  more than $(6k)^{6k}$.
Hence,
\[
\big|\av_{36k^2}(1324)\big| \;\geqs \;
\big|\PPP_k\big| \cdot (6k)^{-6k},
\]
and thus
$81/8$ is a lower bound on the growth rate of $\av(1324)$.
\end{proof}

%
%
%
%
%
%
%
%
%
%
%
%
\section{Domino substructure}\label{secLeavesAndStrips}

To improve the lower bound of Theorem~\ref{thmLowerBound1}, we investigate the structure of dominoes in greater detail.
Specifically we prove two concentration results.
We say that a sequence of random variables $X_1,X_2,\ldots$  \emph{is asymptotically concentrated at} $\mu$ if,
for any $\veps>0$,
for all sufficiently large $n$,
$$
\mathbb{P}\big[\, |X_n-\mu| \:\leqslant\: \varepsilon \,\big] \:\;>\;\: 1-\varepsilon.
$$
We consider two substructures, which we call \emph{leaves} and \emph{empty strips}, definitions of which are given below.
For both, we determine the expected number in an $n$-point domino cell and
establish that their proportion is concentrated at its mean.
As a consequence, almost all dominoes contain ``many'' leaves and ``many'' empty strips.
Thus, when we refine our staircase construction in the next section, we make use of dominoes that have lots of leaves and lots of empty strips.

%
%
%
%
\subsection{Leaves}

Recall that the right-to-left maxima of a permutation are those entries having no larger entry to the right.
Similarly, left-to-right minima are those entries having no smaller entry to the left.
We say that a point in the top, 213-avoiding, cell of a domino is a \emph{leaf} if it is a right-to-left maximum of the permutation.
Analogously, a point in the bottom, 132-avoiding, cell of a domino is a {leaf} if it is a left-to-right minimum of the permutation.
(These correspond to leaves of the acyclic \emph{Hasse graphs} of the cells; see~\cite{BevanAv1324GR, bousquet-melou:forest-like-perms}.)
In Figure~\ref{figArchSystems} on page~\pageref{figArchSystems}, the leaves are shown as rings.

Recall, from Proposition~\ref{propArchSystems}, our bijection between domino cells and arch systems.
Under this bijection, leaves in a 213-avoiding cell correspond exactly to points which are not the left ends of arcs, and
leaves in a 132-avoiding cell correspond to points which are not the right ends of arcs (see Figure~\ref{figArchSystems}).
Thus, adapting Proposition~\ref{propFuncEqArches},
if $A(v,t)=A(z,v,t)$ satisfies the functional equation
\[
	A(v,t)
	\;=\;
	1 \:+\: \frac{z\+t\+A(v,t)}{1-z\+A(v,t)}\:+\:
	z\+(1+v)\+\left(A(v,t)\:+\:\frac{A(v,t)-A(0,t)}{v}\right),
\]
then
$A(0,t)=A(z,0,t)$ is the bivariate generating function for dominoes in which $z$ marks points and $t$ marks leaves
in the top cell.

We want to know how many leaves we can expect to find in a domino cell. We calculate the expected number explicitly.

\begin{prop}\label{propLeaves}
	The total number of leaves in the top cells of all $n$-point dominoes is \[\frac{5(3n+1)!}{(n -1)!(2n + 3)!}.\]
    Consequently, the expected number of leaves in an $n$-point domino is asymptotically $5n/9$.
\end{prop}

In this and subsequent proofs, we use $\partial_xf$ to denote the partial derivative $\partial f/\partial x$.

\begin{proof}
The total number of leaves in the top cells of all $n$-point dominoes is given by the coefficient of $z^n$ in
$\partial_t A(0,t)|_{t=1}$.
To calculate this, we use the same technique as in the proof of Theorem~\ref{thm-twocell}, finding a minimal polynomial $P_1(y,z,t)$ of degree $7$ in $y$ for $A(0,t)$,
that is too long to display here.

Differentiating the equation $0 = P_1(y,z,t)$ with respect to $t$ yields
$0 = P_2(y, \partial_t y, z, t)$,
where $P_2$ is a polynomial.
We wish now to eliminate $y$ from $P_2$ so that a minimal polynomial for $\partial_t A(0,t)$ remains.
This is achieved by computing the resultant of $P_1$ and $P_2$ with respect to their first arguments.
We find that
\[
	\Res\big(P_1(y,z,t),\, P_2(y, y_1, z, t),\, y\big) \;=\; Q(z,t)R(y_1, z, t),
\]
where $Q(z,t)$ is a polynomial only in $z$ and $t$, and $R$ is irreducible.

We conclude therefore that $R(y,z,t)$ is a minimal polynomial for $\partial_t A(0,t)$. Substituting $t=1$ shows that $R(y,z,1)$ factors into two terms, one of which must be a minimal polynomial for $\partial_t A(0,t)|_{t=1}$. By computing initial terms in the power series expansion of the roots of each factor, we
deduce that $\partial_t A(0,t) |_{t=1}$ is a root of
\[
    z^3\+ y^3 + 5\+z^2\+ y^2 + (5\+z-1) y + z .
\]

It can be verified that the coefficient of $z^n$ in the power series expansion of $\partial_t A(0,t) |_{t=1}$ is
\[
	\frac{5(3n+1)!}{(n -1)!(2n + 3)!}  ,
\]
using \emph{Mathematica} as in the proof of Theorem~\ref{thm-twocell}, or otherwise.
Therefore, the expected number of leaves in the top cell of a domino with $n$ points is
\[
	\frac{\;\;\frac{5(3n+1)!}{(n -1)!(2n + 3)!}\;\;}{\frac{2(3n+3)!}{(n + 2)!(2n + 3)!}} \;=\; \frac{5n(n+2)}{6(3n+2)},
\]
from which it follows by symmetry that the expected number of leaves in an $n$-point domino is asymptotically $5n/9$.
\end{proof}

The sequence of coefficients of the power series for $\partial_t A(0,t)|_{t=1}$ is \href{http://oeis.org/A102893}{A102893} in \emph{OEIS}~\cite{OEIS}. This has been shown by Noy~\cite{noy:noncrossing-trees} to count the number of noncrossing trees on a circle with $n+1$ edges and root degree at least 2.
It would be interesting to find a bijection between these objects and the leaves of 1324-avoiding dominoes.

We need to show that the proportion of points that are leaves is asymptotically concentrated.
We calculate the variance directly.

\begin{prop}\label{propLeavesVariance}
  The proportion of leaves in the top cell of an $n$-point domino is asymptotically concentrated at its mean.
\end{prop}

\begin{proof}
Let $E_n$ be the expected number of leaves in the top cell of an $n$-point domino, given by Proposition~\ref{propLeaves},
and let $V_n$ be the variance of the number of leaves in the top cell of an $n$-point domino.
As described in Flajolet and Sedgewick~\cite[Proposition III.2]{flajolet:ac},
\[
	V_n \;=\; \frac{[z^n]\partial_{tt} A(0,t)|_{t=1}}{[z^n] A(0,1)} \:+\: E_n \:-\: {E_n}^{\!2}.
\]
We start by determining $\partial_{tt} A(0,t)|_{t=1}$.
The minimal polynomial for $\partial_{tt} A(0,t)|_{t=1}$ is computed from the minimal polynomial for $\partial_{t} A(0,t)$ using the same method as in the proof of Proposition~\ref{propLeaves}. One finds that $0 = T(\partial_{tt} A(0,t)|_{t=1},z)$, where
\begin{align*}
	T(y,z) \;=\; &\phantom{{}\:+\:{}} z^3(27z-4)(64z^2-31z+4)y^3 \\
     & \:-\: 2z^2(27z-4)(16z^3+39z^2-22z+3)y^2\\
	 & \:+\: 4(36 z^6+186 z^5+118 z^4-243 z^3+102 z^2-17 z+1)y \\
     & \:-\: 8z^2(z^4+8z^3+15z^2-8z+1).
\end{align*}
The coefficient $[z^n]\partial_{tt} A(0,t)|_{t=1}$ is the total number of ordered pairs of distinct leaves in the top cells of $n$-point dominoes. Since this is more  than the total number of leaves and no more than the square of that number, by Proposition~\ref{propLeaves} the dominant singularity of
$\partial_{tt} A(0,t)|_{t=1}$ is $4/27$.

The minimal polynomial $T(y,z)$ allows us to compute the Puiseux expansion of $\partial_{tt} A(0,t)|_{t=1}$ at $z = 4/27$:
\[
	\partial_{tt} A(0,t)|_{t=1} \;=\; \tfrac{25}{144}\+\tfrac{1}{\sqrt{4/27-z}} \:+\: O(1).
\]
It follows from~\cite[Theorem VI.1]{flajolet:ac} that
\[
	[z^n]\partial_{tt} A(0,t)|_{t=1} \;=\; \tfrac{25}{96}\sqrt{\tfrac{3}{\pi}} \left(\tfrac{27}{4}\right)^n n^{-1/2}\left( 1 + O\left(\tfrac{1}{n}\right)\right).
\]	
Using Stirling's Approximation, we find
\[
	[z^n]A(0,1) \;=\; |\D_n| \;=\; \tfrac{2(3n+3)!}{(n + 2)!(2n + 3)!} \;=\; \tfrac{27}{8} \sqrt{\tfrac{3}{\pi}} \left( \tfrac{27}{4} \right)^n n^{-5/2} \left(1 + O\left(\tfrac{1}{n} \right)\right).
\]
Thus,
\[
	\frac{[z^n]\partial_{tt} A(0,t)|_{t=1}}{[z^n] A(0,1)} \;=\; \frac{\frac{25}{96}\sqrt{\frac{3}{\pi}} \left(\frac{27}{4}\right)^n n^{-1/2}\left( 1 + O\left(\frac{1}{n}\right)\right)}{\frac{27}{8} \sqrt{\frac{3}{\pi}} \left( \frac{27}{4} \right)^n n^{-5/2} \left(1 + O\left(\frac{1}{n} \right)\right)} \;=\; \frac{25}{324}n^2 + O(n).
\]
Therefore, the variance is
\[
	\left(\tfrac{25}{324}n^2 + O(n)\!\right) \:+\: \left(\tfrac{5}{18}n + O(1)\!\right) \:-\: \left(\tfrac{25}{324}n^2 + O(n)\!\right) \;=\; O(n).
\]
As the variance is at most linear in $n$, the standard deviation is $O(\sqrt{n})$. Since the order of the standard deviation is strictly smaller than the order of the expected value, by Chebyshev's inequality the proportion of leaves is concentrated at its mean.
\end{proof}

%
%
%
%
\subsection{Empty strips}\label{subsectEmptyStrips}

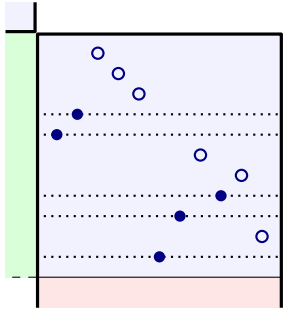
\begin{figure}[ht]
\begin{center}
  \begin{tikzpicture}[scale=0.27,line join=round]
    \fill[blue!5]   (0,0) rectangle (12,12);
    \fill[red!10]   (0,0) rectangle (12,0-1.5);
    \fill[green!15] (0,0) rectangle (0-1.5,12);
    \fill[blue!5]   (0,12) rectangle (0-1.5,12+1.5);
    \draw[very thick] (0+.06,0-1.5)--(0+.06,12-.06)--(12-.06,12-.06)--(12-.06,0-1.5);
    \draw[very thick] (0-1.5,12+.06)--(0-.06,12+.06)--(0-.06,12+1.5);
    \draw[]           (0,0)--(12,0);
    \draw[dashed]     (0,0)--(0-1.5,0);
    \foreach \y in {1,3,4,7,8}
      \draw[thick,dotted] (00,\y) -- (12,\y);
    \plotpermnobox[blue!50!black]{}{7,8,0,0,0,1,3,0,4,0,0}
    \setplotptfun{\draw}
    \plotpermnobox[blue!50!black,thick]{}{0,0,11,10,9,0,0,6,0,5,2}
    \setplotptfun{\fill}
  \end{tikzpicture}
\end{center}
\caption{Strips in a domino cell}\label{figStrips}
\end{figure}

In a vertical domino, we consider the cells to be divided into horizontal \emph{strips} by their non-leaf points.
For example, in Figure~\ref{figStrips} the cell is divided into six horizontal strips by its five non-leaf points.
We are interested in the number of such strips which contain no leaves, which we call \emph{empty strips}. In Figure~\ref{figStrips} there are three empty strips.

By the bijection between domino cells and arch systems
in Proposition~\ref{propArchSystems}, empty strips
in a 213-avoiding cell
correspond to arcs to the left of
points that are both the left and right endpoint of an arc (see Figure~\ref{figArchSystems}).
An empty strip is also possible at the bottom of the cell (but not at the top, the uppermost point always being leaf).\label{commentEdgeEmptyStrips}
This possibility does not affect the asymptotics, so we just count \emph{medial} empty strips.

Thus, adapting Proposition~\ref{propFuncEqArches},
if $A(v,s)=A(z,v,s)$ satisfies the functional equation
\[
	A(v,s)
	\;=\;
	1 \:+\: z\+A(v,s) \:+\: \frac{z^2\+A(v,s)^2}{1-z\+s\+A(v,s)}\:+\:
	z\+(1+v)\+\left(A(v,s)\:+\:\frac{A(v,s)-A(0,s)}{v}\right),
\]
then
$A(0,s)=A(z,0,s)$ is the bivariate generating function for dominoes in which $z$ marks points and $s$ marks medial empty strips in the top cell.

How many empty strips can we expect to find in a domino cell? We calculate the expected number exactly.

\begin{prop}\label{propEmptyStrips}
  	The total number of medial empty strips in the top cells of all $n$-point dominoes is \[\frac{10(3n)!}{(n - 3)!(2n + 4)!}.\]
    Consequently, the expected number of empty strips in an $n$-point domino cell is asymptotically $5n/27$.
\end{prop}

\begin{proof}
  The total number of medial empty strips in the top cells of all $n$-point dominoes is given by the coefficient of $z^n$ in
$\partial_s A(0,s)|_{s=1}$.
  Using the same approach as in the proof of Proposition~\ref{propLeaves}, we can deduce that $\partial_s A(0,s)|_{s=1}$ is a root of the equation
  \[
    z^4 \+ y^3
    - (15\+ z + 2) z^2 \+ y^2
    - (10\+ z^3 - 25\+ z^2 + 10\+ z - 1) y
    - z^3
    ,
  \]
and verify that the coefficient of $z^n$ in the power series expansion of $\partial_s A(0,s) |_{s=1}$ is exactly
\[
	\frac{10(3n)!}{(n - 3)!(2n + 4)!}  .
\]
Therefore, the expected number of medial empty strips in the top cell of a domino with $n$ points is
\[
	\frac{\;\;\frac{10(3n)!}{(n -3)!(2n + 4)!}\;\;}{\frac{2(3n+3)!}{(n + 2)!(2n + 3)!}} = \frac{5n(n-1)(n-2)}{6(3n+1)(3n+2)},
\]
from which it follows by symmetry that the expected number of empty strips in an $n$-point domino is asymptotically $5n/27$.
\end{proof}
The sequence of coefficients of the power series of $\partial_s A(0,s)|_{s=1}$ is \href{http://oeis.org/A233657}{A233657} in \emph{OEIS}~\cite{OEIS}. These are the two-parameter Fuss--Catalan (or Raney) numbers with parameters $p=3$ and $r=10$.
It would be interesting to find a bijection between medial empty strips in 1324-avoiding dominoes
and some other combinatorial class enumerated by this sequence.

Again, we need a concentration result, so we determine the variance.

\begin{prop}\label{propStripsVariance}
  The proportion of empty strips in the top cell of an $n$-point domino is asymptotically concentrated at its mean.
\end{prop}

\begin{proof}
As before, the minimal polynomial for $\partial_{ss} A(0,s)|_{s=1}$ is computed from the minimal polynomial for $\partial_{s} A(0,s)$. It is a root of the cubic
\begin{align*}
&\phantom{{}\:+\:{}}  z^4 (27 z-4) (64 z^2-31 z+4) y^3 \\
&\:-\: 2 z^2 (27 z-4) (64 z^4-1388 z^3+534 z^2-23 z-8) y^2 \\
&\:-\: 4 (1536 z^8-22676 z^7+82275 z^6-112651 z^5+72411 z^4-24430 z^3+4471 z^2-421 z+16) y \\
&\:-\: 8 z^4 (64 z^5-719 z^4+1371 z^3-918 z^2+213 z-16)  .
\end{align*}
This allows us to compute the Puiseux expansion of $\partial_{ss} A(0,s)|_{s=1}$ at $z = 4/27$:
\[
	\partial_{ss} A(0,s)|_{s=1} \;=\; \tfrac{25}{1296}\+\tfrac{1}{\sqrt{4/27-z}} \:+\: O(1).
\]
It follows that
\[
	[z^n]\partial_{ss} A(0,s)|_{s=1} \;=\; {\tfrac{25}{864} \sqrt{\tfrac{3}{\pi}} \left( \tfrac{27}{4} \right)^n n^{-1/2}} \left( 1 + O\left(\tfrac{1}{n}\right)\right) .
\]	
Thus,
\[
	\frac{[z^n]\partial_{ss} A(0,s)|_{s=1}}{[z^n] A(0,1)} \;=\; \frac{\frac{25}{864} \sqrt{\frac{3}{\pi}} \left( \frac{27}{4} \right)^n n^{-1/2}\left(1 + O\left(\frac{1}{n} \right)\right)}{\frac{27}{8} \sqrt{\frac{3}{\pi}} \left( \frac{27}{4} \right)^n n^{-5/2} \left(1 + O\left(\frac{1}{n} \right)\right)} \;=\; \frac{25}{2916}n^2 + O(n) .
\]
Therefore, the variance is
\[
	\left(\tfrac{25}{2916}n^2 + O(n)\!\right) \:+\: \left(\tfrac{5}{54}n + O(1)\!\right) \:-\: \left(\tfrac{25}{2916}n^2 + O(n)\!\right) \;=\; O(n).
\]
The result follows by Chebyshev's inequality.
\end{proof}

%
%
%
%
\subsection{Dominoes with many leaves and many empty strips}

As a consequence of these concentration results,
sets of dominoes with many leaves and many empty strips have the same growth rate as the set of all dominoes.
For $\alpha,\beta\in[0,1]$,
let $\D^{\alpha,\beta}_{n}$ be the set of $n$-point dominoes
with at least $\alpha n/2$ leaves in each cell and at least $\beta n/2 +1$ empty strips in each cell.
Let $\D^{\alpha,\beta}=\bigcup_{n}\D^{\alpha,\beta}_{n}$.
The use of $\beta n/2 +1$, rather than $\beta n/2$, is explained in the proof of Proposition~\ref{propBalancedDominoes2} below.

\begin{cor}\label{corLeafyStrippyDominoesGrowth}
  If $\alpha<5/9$ and $\beta<5/27$, then $\gr(\D^{\alpha,\beta})=27/4$.
\end{cor}
\begin{proof}
  By Propositions~\ref{propLeaves} and~\ref{propLeavesVariance}, for sufficiently large $n$,
  at least four fifths of $n$-point dominoes have $\alpha n/2$ or more leaves in their top cell, and, by symmetry,
  at least four fifths have $\alpha n/2$ or more leaves in their bottom cell.
  Let $\beta'$ be in the open interval $(\beta,5/27)$. Then,
  for sufficiently large $n$, we have $\beta' n/2\geqs \beta n/2+1$ and so, applying Propositions~\ref{propEmptyStrips} and~\ref{propStripsVariance} with $\beta'$,
  at least four fifths of $n$-point dominoes have $\beta n/2+1$ or more empty strips in their top cell, and
  at least four fifths have $\beta n/2+1$ or more empty strips in their bottom cell.
  Hence asymptotically, at least one fifth of all dominoes are in $\D^{\alpha,\beta}$. The result follows from Theorem~\ref{thm-twocell}.
\end{proof}

An analogous result holds for sets of \emph{balanced} dominoes with many leaves and many empty strips.
Let $\B_m^{\alpha,\beta}$ be the set of $2m$-point balanced dominoes, with at least $\alpha m$ leaves in each cell
and at least $\beta m+1$ empty strips in each cell, and let $\B^{\alpha,\beta}=\bigcup_{m}\B^{\alpha,\beta}_m$.

\begin{prop}\label{propBalancedDominoes2}
If $\alpha<5/9$ and $\beta<5/27$, then $\gr(\B^{\alpha,\beta})=27/4$.
\end{prop}

The proof mirrors that of Proposition~\ref{propBalancedDominoes}.

\begin{proof}
For suitable values of the parameters,
let $\LLL(t,b,\ell_{\textsf{\tiny T}},\ell_{\textsf{\tiny B}},e_{\textsf{\tiny T}},e_{\textsf{\tiny B}})$ denote the set of $(t+b)$-point dominoes with $t$ points in the top cell, $b$ points in the bottom cell, $\ell_{\textsf{\tiny T}}$ leaves in the top cell, $\ell_{\textsf{\tiny B}}$~leaves in the bottom cell, $e_{\textsf{\tiny T}}$ empty strips in the top cell and $e_{\textsf{\tiny B}}$ empty strips in the bottom cell.
For a given $m$, let $\LLL_m$ be some such set 
whose size is maximal subject to the conditions $t+b=m$,
$\ell_{\textsf{\tiny T}},\ell_{\textsf{\tiny B}}\geqs\alpha m/2$ and $e_{\textsf{\tiny T}},e_{\textsf{\tiny B}}\geqs\beta m/2+1$.
Note that $\LLL_m\subseteq \D^{\alpha,\beta}_{m}$.
Since $0\leqs t,\ell_{\textsf{\tiny T}},\ell_{\textsf{\tiny B}},e_{\textsf{\tiny T}},e_{\textsf{\tiny B}} \leqs m$, there are at most $(m+1)^5$ possible choices for the parameters. Hence by the pigeonhole principle,
	\[
		|\LLL_m| \;\geqs\; \frac{|\D^{\alpha,\beta}_{m}|}{(m+1)^5}.
	\]
Let $\sigma$ and $\tau$ be any two $m$-point dominoes from $\LLL_m$.
Consider the domino $\rho=\sigma \varoast \overset{\curvearrowleft}{\tau}$, whose arch configuration is constructed by concatenating the arch configuration of $\sigma$ and the arch configuration of the $180^\circ$ rotation of $\tau$.
This domino has $m$ points in each cell and $\ell_{\textsf{\tiny T}}+\ell_{\textsf{\tiny B}}\geqs\alpha m$ leaves in each cell.
If the top cell of $\sigma$ has an empty strip at the bottom, then this combines with the non-empty strip at the bottom of the bottom cell of $\tau$, in which case the top cell of $\rho$ has $e_{\textsf{\tiny T}}+e_{\textsf{\tiny B}}-1$ empty strips. Otherwise it has $e_{\textsf{\tiny T}}+e_{\textsf{\tiny B}}$ empty strips. In either case, this is at least $\beta m+1$. An analogous argument applies to the bottom cell, so $\rho$ is a balanced domino in $\B^{\alpha,\beta}_m$.

Moreover, $\sigma$ and $\tau$ can be recovered from $\rho$ simply by splitting its arch configuration into two halves. Thus,
\[
    |\B^{\alpha,\beta}_{m}| \;\geqs\; |\LLL_m|^2 \;\geqs\; \frac{|\D^{\alpha,\beta}_m|^{2}}{(m+1)^{10}}.
\]
Since it is also the case that $|\D^{\alpha,\beta}_{2m}|\geqs|\B^{\alpha,\beta}_m|$, it follows, by taking the $2m$th root, and the limit as $m$ tends to infinity, that $\gr(\B^{\alpha,\beta}) = \gr(\D^{\alpha,\beta}) = 27/4$.
\end{proof}

%
%
%
%
%
%
%
%
%
%
%
%
\section{A better lower bound}\label{secLowerBound2}

\begin{figure}[ht]
\begin{center}
\begin{tikzpicture}[scale=1.25]
\foreach \x [evaluate=\x as \xx using int(2*\x)] in {1,2,4} {
  \draw[fill=red!10] (\x,-\x) rectangle (\x+1,1-\x);
  {\node at (\x+0.5, 0.5-\x) {\scriptsize $\av(132)$};}
  \node at (\x+0.840,0.825-\x) {\scriptsize \xx};
}
\fill[red!5] (5,-4) rectangle (6,-4.2);
\foreach \x [evaluate=\x as \xx using int(2*\x-1)] in {1,3,4} {
  \draw[fill=blue!5] (\x,1-\x) rectangle (\x+1,2-\x);
  {\node at (\x+0.5, 1.5-\x) {\scriptsize $\av(213)$};}
  \node at (\x+0.840, 1.825-\x) {\scriptsize \xx};
}
\fill[blue!5] (6,-4) rectangle (6.2,-4.2);
\foreach \x [evaluate=\x as \xx using int(2*\x)] in {3} {
  \draw[dashed,fill=green!15] (\x,-\x) rectangle (\x+1,1-\x);
  {\node at (\x+0.5, 0.5-\x) {\scriptsize $\av(132)$};}
  \node at (\x+0.840, 0.825-\x) {\scriptsize \xx};
}
\foreach \x [evaluate=\x as \xx using int(2*\x-1)] in {2,5} {
  \draw[dashed,fill=yellow!10] (\x,1-\x) rectangle (\x+1,2-\x);
  {\node at (\x+0.5, 1.5-\x) {\scriptsize $\av(213)$};}
  \node at (\x+0.840, 1.825-\x) {\scriptsize \xx};
}
  \draw[very thick] (2-0.01,1.01-2) rectangle (2-0.99,2.99-2);
  \draw[very thick] (0.01+2,0.99-2) rectangle (2+1.99,0.01-2);
  \draw[very thick] (5-0.01,1.01-5) rectangle (5-0.99,2.99-5);
  \draw[          ] (6,-4)--(6,-4.2);
  \draw[very thick] (5.01,-4.2)--(5.01,0.99-5)--(6.2,0.99-5);
  \foreach \z in {2,5} {
    \foreach \y in {0.1,0.15,...,0.91}
	  \draw (\z-0.08,\y-\z+1) -- (\z+0.06,\y-\z+1);
  }
  \foreach \z in {4} {
    \foreach \y in {0.1,0.15,...,0.91}
	  \draw (\z-0.06,\y-\z+1) -- (\z+0.08,\y-\z+1);
  }
\end{tikzpicture}
\end{center}
\caption{The decomposition of the staircase into dominoes and connecting cells}
\label{figLowerBoundStaircase2}
\end{figure}
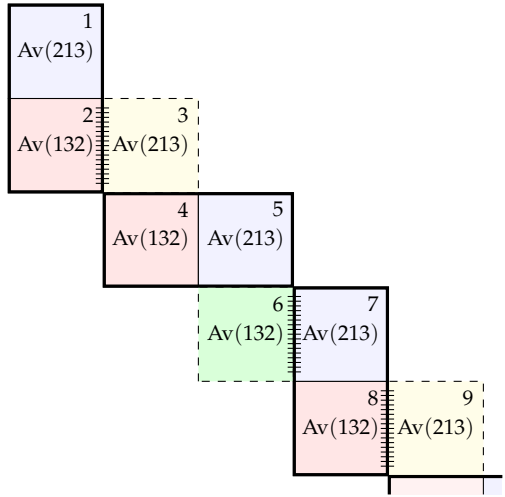

In this final section, we modify the construction used to prove Theorem~\ref{thmLowerBound1} to yield an improved lower bound.
We make use of exactly the same decomposition of the staircase, which we reproduce here in Figure~\ref{figLowerBoundStaircase2}.
However, we change the rules concerning the permitted interleaving of points between the cells.
We also exploit the additional properties of dominoes established in Section~\ref{secLeavesAndStrips}.

\begin{figure}[ht]
\begin{center}
  \tikzset{brace/.style= {decoration={brace, mirror, amplitude=4pt}, decorate}}
  $\qquad$
  \begin{tikzpicture}[scale=0.24]
    \fill[fill=red!10] (-1,19.5) rectangle (0.5,36.5);
    \fill[blue!5] (0.5,19.5) rectangle (18.5-.1,36.5);
    \fill[blue!5] (18.5,1.5) rectangle (35.5,19.5-.1);
    \fill[fill=red!10] (18.5,0) rectangle (35.5,1.5);
    \draw[dashed,fill=green!15] (0.5,1.5) rectangle (18.5,19.5);
    \draw[fill=green!30] (2.5,15.5) rectangle (5.5,19.4);
    \draw[fill=green!30] (6.5,14.5) rectangle (7.5,15.5);
    \draw[fill=green!30] (9.5,11.5) rectangle (10.5,12.5);
    \draw[fill=green!30] (10.5,3.5) rectangle (13.5,8.5);
    \draw[fill=green!30] (16.5,1.6) rectangle (18.4,3.5);
    \draw[very thick] (-1,19.5)--(18.5-.1,19.5)--(18.5-.1,36.5)--(-1,36.5);
    \draw[] (0.5,19.5)--(0.5,36.5);
    \draw[very thick] (18.5,0)--(18.5,19.5-.1)--(35.5,19.5-.1)--(35.5,0);
    \draw[] (18.5,1.5)--(35.5,1.5);
    \foreach \x in {1,2,6,8,9,14,15,16}
      \draw[dotted] (\x,1.5) -- (\x,36.5);
    \foreach \y in {9,11,13}
      \draw[thick,dotted] (0.5,\y) -- (35.5,\y);
    \draw[very thin] (24-.25,17) -- (.5,17);
    \draw[very thin] (26-.25,14) -- (.5,14);
    \draw[very thin] (30-.25,10) -- (.5,10);
    \draw[very thin] (32-.25, 7) -- (.5, 7);
    \draw[very thin] (34-.25, 5) -- (.5, 5);
    \plotpermnobox[ blue!50!black]{}{35,25, 0, 0, 0,27, 0,31,33, 0, 0, 0, 0,29,21,23, 0, 0, 0,11, 0,13, 0, 0, 0, 0, 0, 9, 0, 0, 0, 0, 0, 0}
    \setplotptfun{\draw}
    \plotpermnobox[ blue!50!black,thick]{}{ 0, 0, 0, 0, 0, 0, 0, 0, 0, 0, 0, 0, 0, 0, 0, 0, 0, 0, 0, 0, 0, 0, 0,17, 0,14, 0, 0, 0,10, 0, 7, 0, 5}
    \setplotptfun{\fill}
    \plotpermnobox[green!25!black]{}{ 0, 0,18,16,19, 0,15, 0, 0,12, 4, 6, 8, 0, 0, 0, 2, 3}
    \draw[brace] (35.7,13.05) -- (35.7,19.35) node[midway, label=right:{$a_0 = 2$}]{};
    \draw[brace] (35.7,11.05) -- (35.7,12.95) node[midway, label=right:{$a_1 = 0$}]{};
    \draw[brace] (35.7, 9.05) -- (35.7,10.95) node[midway, label=right:{$a_2 = 1$}]{};
    \draw[brace] (35.7, 1.55) -- (35.7, 8.95) node[midway, label=right:{$a_3 = 2$}]{};
  \end{tikzpicture}
\end{center}
  \caption{Interleaving the points in a connecting cell with those in two domino cells}
  \label{figLeafInterleaving}
\end{figure}
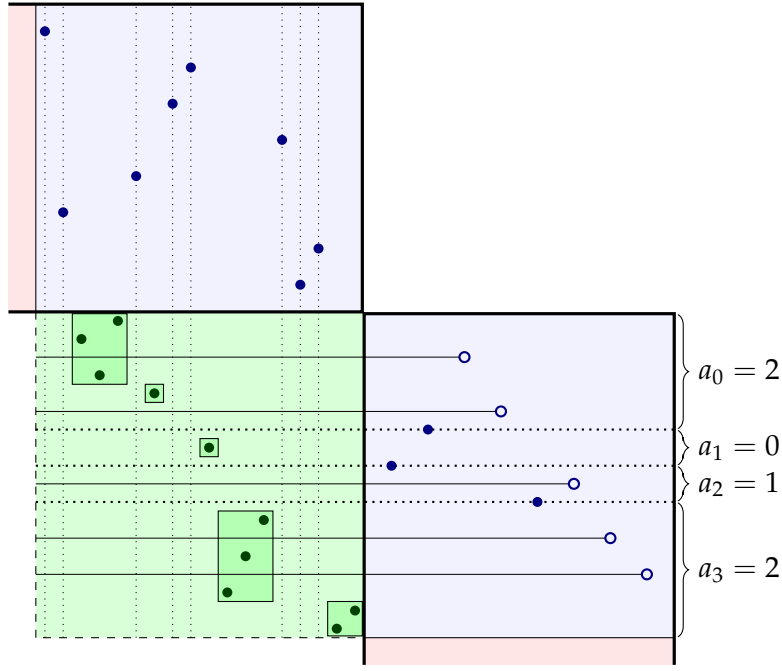

Recall that in our earlier construction, we
ensure that there is no occurrence of 1324 by requiring every point in a domino cell to be positioned between the components in the adjacent connecting cells, as illustrated in
Figure~\ref{figInterleaving}.
For our improved lower bound, we relax this restriction in the case of domino cells \emph{to the left or right of a connecting cell}. In this case,
we require only that \emph{non-leaves} in a domino cell are positioned between the components. Leaves may be positioned arbitrarily.
See Figure~\ref{figLeafInterleaving} for an illustration of a 132-avoiding connecting cell and its adjacent domino cells.
In the domino cell to the right, leaves are shown as rings and non-leaves as disks.

This still prevents any occurrence of 1324.
For example, if a domino cell is to the right of a connecting cell, then this restriction ensures that in any
occurrence of 132 with the 13 in the connecting cell and the 2 in the domino cell the 2 is a leaf, so there can be no point to its upper right to complete a 1324.
In Figure~\ref{figLowerBoundStaircase2}, this greater freedom is shown using small lines between connecting cells and horizontally adjacent domino cells.
Observe that this flexibility only applies to the cells of \emph{vertical} dominoes in the decomposition.
We could similarly
relax the restriction in the case of domino cells above and below a connecting cell. However, this results in a structure we have been unable to analyse.

This refined construction enables us to establish an improved lower bound on the growth rate of 1324-avoiders.

\begin{theorem}\label{thmLowerBound2}
The growth rate of
$\av(1324)$ is at least $10.271012$. 
\end{theorem}

%
%
%
%
\subsection{Horizontally interleaved connecting cells}

Let us consider how a connecting cell can be interleaved with a horizontally adjacent domino cell.
We want to enumerate diagrams like the lower two cells of Figure~\ref{figLeafInterleaving}, where the points in the domino cell at the right have been erased, but the horizontal lines, solid for leaves and dotted for non-leaves, have been retained to record the positions of the points relative to the points in the connecting cell.
Let us call these configurations \emph{horizontally interleaved connecting cells}.

We begin with the generating function for connecting cells, 
\begin{equation}\label{eqHzq}
H(z,q) \;=\; \frac{1}{1-q\+Q(z)} \;=\; \frac{2}{2-q+q\sqrt{1-4z}},
\end{equation}
where $z$ marks points, $q$ marks components, and
$Q(z)=\frac{1}{2} (1-\sqrt{1-4 z})$ is the generating function for components of a connecting cell.

As described in Section~\ref{subsectEmptyStrips}, the non-leaves of a vertical domino cell divide it and the adjacent connecting cell into horizontal strips.
Suppose that such a domino cell
has $\ell$ leaves and $r$ non-leaves.
The $r$ non-leaves divide the cell into $r+1$ horizontal strips, each containing a certain number of leaves.
Let $a_i$ denote the number of leaves in the $i$th strip from the top, for $i=0,\ldots,r$, so $a_0+\ldots+a_r=\ell$.
See Figure~\ref{figLeafInterleaving} for an illustration.

The generating function for the possibilities in the $i$th strip is given by
\begin{equation}\label{eqHjzq}
H_{a_i}(z,q) \;=\; \Omega_{a_{i}}[H(z,q)],
\end{equation}
where each $\Omega_{j}$ is a linear operator given by
\[
\Omega_{j}[z^{n}] \;=\; \binom{n+j}{j}z^n ,
\]
or equivalently,
\begin{equation}\label{eqOmega}
\Omega_{j}[F(z)] \;=\; \frac{1}{j!}\frac{\partial^{j}}{\partial z^{j}}\big(z^{j}F(z)\big).
\end{equation}
Hence, for a fixed sequence $(a_i)_{i=0}^r$ of strip sizes, the generating function for
horizontally interleaved connecting cells,
counting once each possible way of interleaving with the contents of the horizontally adjacent domino cell, is given by
\begin{equation}\label{eqInterleaving}
\prod_{i=0}^{r}H_{a_{i}}(z,q).
\end{equation}
We cannot work directly with this expression, since it would require us to keep track of all the strip sizes.
So, in order to establish a lower bound, we seek to minimise the above expression over all sequences $a_0,\dots,a_r$ such that $a_0+\ldots+a_r=\ell$.
With the next two propositions we demonstrate that such a minimum exists for any fixed $r$ and $\ell$, in the sense that every coefficient of~\eqref{eqInterleaving} is minimised for the same sequence $a_{0},\ldots,a_{r}$.
More specifically, we prove that this minimum occurs when no two terms of the sequence differ by more than 1.
We call such a sequence \emph{equitable}.

In our refinement of the staircase, a certain number of the strips are required to be empty. With this additional requirement, for a lower bound, we thus need an equitable distribution of the leaves among the rest of the strips.

The following proposition is framed in the general setting of partially ordered rings, though for our purposes these are always rings of formal power series with real coefficients. Recall that a \emph{partially ordered ring} $(R,\leqs)$, is a (commutative) ring $R$ together with a partial order $\leqs$ on the elements of $R$ such that if $a,b,c\in R$ then $a\leqs b$ if and only if $a+c\leqs b+c$, and $a,b\geqs0$ implies $ab\geqs0$. Given such a ring $(R,\leqs)$, we define $(R[[q]],\leqs)$ to be the ring of formal power series over $R$ equipped with the partial order defined by $h(q)\geqs0$ if and only if every coefficient of $h(q)$ is in $R_{\geqs0} = \{r\in R: r\geqs 0\}$.

A sequence $a_{0},a_{1},\ldots$ in $(R,\leqs)$ is \emph{log-convex} if, for every pair of integers $i,j$ with $0\leqs i<j$, we have $a_{i}a_{j+1}\geqs a_{i+1}a_{j}.$

\begin{prop}\label{propLogConvex}
Let $(R,\leqs)$ be a partially ordered ring and let $a_{0},a_{1},\ldots$ be a log-convex sequence in $R_{\geqs0}$. Furthermore, let $F(z)=a_{0}+a_{1}z+\ldots$ be the generating function of this sequence. Then the sequence $\Omega_{0}[F(z)], \Omega_{1}[F(z)],\ldots$ is log-convex in the partially ordered ring $(R[[z]],\leqs)$.
\end{prop}

\begin{proof}
We just need to show that for each $k\geqs0$ and each $a>b\geqs0$,
\[
[z^k] \Big(\Omega_{a+1}[F(z)]\+\Omega_{b}[F(z)] \:-\:  \Omega_{a}[F(z)]\+\Omega_{b+1}[F(z)]\Big) \;\geqs\; 0.
\]
This coefficient can be computed as
\begin{align*}
       &\sum_{j=0}^{k}\binom{a+1+j}{j}a_{j}\binom{b+k-j}{k-j}a_{k-j} \:-\: \sum_{j=0}^{k}\binom{a+j}{j}a_{j}\binom{b+1+k-j}{k-j}a_{k-j}\\
\;=\;\;&\sum_{j=0}^{k}\left(\binom{a+1+j}{j}\binom{b+k-j}{k-j}\:-\:\binom{b+1+j}{j}\binom{a+k-j}{k-j}\right)a_{j}a_{k-j}\\
\;=\;\;&\sum_{j=0}^{k}\sum_{i=0}^{j}\left(\binom{a+i}{i}\binom{b+k-j}{k-j}\:-\:\binom{b+i}{i}\binom{a+k-j}{k-j}\right)a_{j}a_{k-j}\\
\;=\;\;&\sum_{j=0}^{k}\sum_{i=0}^{k-j}\left(\binom{a+i}{i}\binom{b+j}{j}\:-\:\binom{b+i}{i}\binom{a+j}{j}\right)a_{j}a_{k-j}\\
\;=\;\;&\sum_{j=0}^{k}\sum_{i=0}^{\min(j-1,k-j)}\left(\binom{a+i}{i}\binom{b+j}{j}\:-\:\binom{b+i}{i}\binom{a+j}{j}\right)(a_{j}a_{k-j}-a_{i}a_{k-i}).
\end{align*}
Now, the coefficient of $a_ja_{k-j}-a_ia_{k-i}$ in each summand, namely
\[
\binom{a+i}{i}\binom{b+j}{j}\:-\:\binom{b+i}{i}\binom{a+j}{j},
\]
is negative, since
$i<j$ and $a>b$.
Also, since $i\leqs  k-j$, we have $a_{i}a_{k-i}\geqs a_{j}a_{k-j}$. Hence each summand is nonnegative
and the entire sum is positive, which implies that the sequence $\Omega_{0}[F(z)], \Omega_{1}[F(z)],\ldots$ is log-convex in $(R[[z]],\leqs)$.
\end{proof}

We now apply this to the enumeration of horizontally interleaved connecting cells.

\begin{prop}\label{propHzq}
Let
\[
H(z,q) \;=\; \frac{2}{2-q+q\sqrt{1-4z}} \;=\; h_{0}(q)+zh_{1}(q)+z^{2}h_{2}(q)+\ldots
\]
be the generating function for connecting cells where $z$ marks points and $q$ marks components.
Then the sequence of polynomials $h_{0}(q),h_{1}(q),\ldots$ is log-convex in $(\mathbb{R}[[q]],\leqs)$.
Consequently,
the sequence $H(z,q),\,H_1(z,q),\,H_2(z,q),\,\ldots$ is log-convex in $(\mathbb{R}[[z,q]],\leqs)$.
\end{prop}

\begin{proof}
Since the generating function $H(z,q)$ satisfies the equation
\[
H(z,q) \;=\; 1+z\frac{q^{2}H(z,q)-qH(z,1)}{q-1},
\]
it follows that for each $i\geqs 1$,
\[
h_{i}(q) \;=\; \frac{q^{2}h_{i-1}(q)-qh_{i-1}(1)}{q-1} \;=\; \frac{q^{2}h_{i-1}(q)-qc_{i-1}}{q-1},
\]
where $c_n=\binom{2n}{n}/(n+1)$ is the $n$th Catalan number.
Rearranging this gives the equation
\[
h_{i-1}(q) \;=\; \frac{(q-1)h_{i}(q)+qc_{i-1}}{q^2}.
\]
We need to prove that if $j>i\geqs 1$ then we have $h_{i-1}(q)h_{j}(q)\geqs h_{i}(q)h_{j-1}(q)$.
This happens if and only if
\[
\frac{(q-1)h_{i}(q)+qc_{i-1}}{q^2}h_{j}(q) \;\geqs\; \frac{(q-1)h_{j}(q)+qc_{j-1}}{q^2}h_{i}(q) ,
\]
which simplifies to
\[
c_{i-1}h_{j}(q)\:-\: c_{j-1}h_{i}(q) \;\geqs\; 0.
\]
One can easily prove by induction, or otherwise, that
\[
h_{i}(q) \;=\; \sum_{k=1}^{i} h_{i,k}\+q^k,
\text{~~where~}
h_{i,k} \:=\: \frac{k}{2i-k}\binom{2i-k}{i}.
\]
It suffices to demonstrate that $c_{i-1}h_{j,k} - c_{j-1}h_{i,k} \geqs 0$ whenever $j>i\geqs k\geqs1$.
By transitivity, we only need consider the case $j=i+1$, when it is readily confirmed that the required inequality holds:
\[
c_{i-1}h_{i+1,k} \:-\: c_{i}h_{i,k} \;=\; \frac{k(k-1)(k-2)\+(2i-2)!\+(2i-k-1)!}{(i+1)!\,i!\,(i-1)!\+(i-k+1)!} \;\geqs\; 0,
\text{~~if~} i\geqs k\geqs1 .
\]
Hence, the sequence $h_{0}(q),h_{1}(q),\ldots$ is log-convex in $(\mathbb{R}[[q]],\leqs)$.
Consequently, by Proposition~\ref{propLogConvex}, the sequence $H(z,q),\,H_1(z,q),\,H_2(z,q),\,\ldots$ is log-convex in $(\mathbb{R}[[z,q]],\leqs)$.
\end{proof}

Thus, as claimed above, among all sequences $a_{0},\ldots, a_{r}$ which satisfy $a_{0}+\ldots+a_{r}=\ell$, the minimum value of every coefficient
of
\[
\prod_{i=0}^{r} H_{a_{i}}(z,q)
\]
is achieved by \emph{equitable} sequences, that is
in which $|a_i-a_j|\leqs 1$ for every $i,j\in\{0,\ldots,r\}$.
This, therefore, is what we apply to the non-empty strips to give a lower bound for the number of horizontally interleaved connecting cells.

%
%
%
%
\subsection{Refining the staircase}

\begin{figure}[ht]
\begin{center}
\begin{tikzpicture}[scale=1.32]
\foreach \x [evaluate=\x as \xx using int(2*\x)] in {1,2,4} {
  \fill[fill=red!10] (\x,-\x) rectangle (\x+1,1-\x);
}
\fill[red!5] (5,-4) rectangle (6,-4.2);
\foreach \x [evaluate=\x as \xx using int(2*\x-1)] in {1,3,4} {
  \fill[fill=blue!5] (\x,1-\x) rectangle (\x+1,2-\x);
}
\fill[blue!5] (6,-4) rectangle (6.2,-4.2);
\node at (1.5,0) {\Large $\B^{\alpha,\beta}_{m}$};
\node at (3,-1.5) {\Large $\B_{\ceil{\gamma m}}$};
\node at (4.5,-3) {\Large $\B^{\alpha,\beta}_{m}$};
\foreach \x [evaluate=\x as \xx using int(2*\x)] in {3} {
  \draw[dashed,fill=green!15] (\x,-\x) rectangle (\x+1,1-\x);
  {\node at (\x+0.5, 0.7-\x) {\scriptsize $\av(132)$};}
   \node at (\x+0.5, 0.4-\x) {\scriptsize $c_m$};
   \node at (\x+0.5, 0.2-\x) {\scriptsize comps};
}
\foreach \x [evaluate=\x as \xx using int(2*\x-1)] in {2,5} {
  \draw[dashed,fill=yellow!10] (\x,1-\x) rectangle (\x+1,2-\x);
  {\node at (\x+0.5, 1.7-\x) {\scriptsize $\av(213)$};}
   \node at (\x+0.5, 1.4-\x) {\scriptsize $c_m$};
   \node at (\x+0.5, 1.2-\x) {\scriptsize comps};
}
  \draw[very thick] (2-0.01,1.01-2) rectangle (2-0.99,2.99-2);
  \draw[very thick] (0.01+2,0.99-2) rectangle (2+1.99,0.01-2);
  \draw[very thick] (5-0.01,1.01-5) rectangle (5-0.99,2.99-5);
  \draw[          ] (6,-4)--(6,-4.2);
  \draw[very thick] (5.01,-4.2)--(5.01,0.99-5)--(6.2,0.99-5);
  \foreach \z in {2,5} {
    \foreach \y in {0.1,0.15,...,0.91}
	  \draw (\z-0.08,\y-\z+1) -- (\z+0.06,\y-\z+1);
  }
  \foreach \z in {4} {
    \foreach \y in {0.1,0.15,...,0.91}
	  \draw (\z-0.06,\y-\z+1) -- (\z+0.08,\y-\z+1);
  }
\end{tikzpicture}
\end{center}
\caption{The scheme used to calculate the improved lower bound}
\label{figLowerBoundStaircase3}
\end{figure}
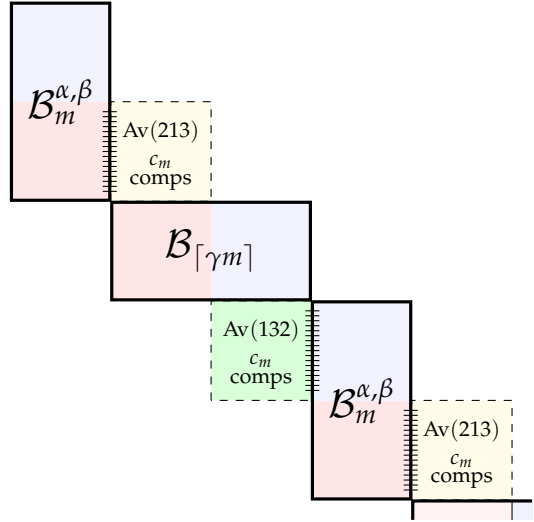

We are now ready to describe more precisely how we modify our construction so as to yield an improved lower bound.
This description is accompanied by Figure~\ref{figLowerBoundStaircase3}.
Recall that $\B^{\alpha,\beta}=\bigcup_{m}\B^{\alpha,\beta}_m$, where $\B^{\alpha,\beta}_m$ consists of dominoes in which each cell has
$m$ points, at least $\alpha m$ leaves
and at least $\beta m+1$ empty strips.
Let $\alpha,\beta>0$
be sufficiently small that $\B^{\alpha,\beta}$
has exponential growth rate $27/4$. By Proposition~\ref{propBalancedDominoes2}, we may choose any $\alpha<5/9$ and $\beta<5/27$.
We also require that $\alpha\geqs11/20$ and $\beta\geqs7/40$.

For fixed values of parameters $\alpha$, $\beta$, $\gamma$ and $\kappa$,
and sufficiently large $k$ and $m$,
let $\PPP_{k,m}$ be the set of gridded permutations, gridded in the first~$6k+2$ cells of the staircase, satisfying the following conditions.
\vspace{-9pt}\begin{itemize}\itemsep0pt
  \item Each non-leaf in a cell of a vertical domino is positioned between components of the horizontally adjacent connecting cell.
  \item Each point in a cell of a horizontal domino is positioned between components of the vertically adjacent connecting cell.
  \item Each vertical domino is an element of $\B^{\alpha,\beta}_m$.
  \item Each horizontal domino is a balanced domino with $\ceil{\gamma m}$ points in each cell, for some $\gamma>0$ to be chosen later.
  \item Each connecting cell has $c_m$ components, where $\liminfty[m]c_m/m=\kappa$, for some $\kappa>0$; the value of $\kappa$ and the sequence $(c_m)$ are to be chosen later.
\end{itemize}\vspace{-9pt}
Note that each domino cell contains a fixed number of points (either $m$ or $\ceil{\gamma m}$). However, the number of points in a connecting cell is not fixed, although its number of skew indecomposable components, $c_m$, is.

We begin by establishing a lower bound for the enumeration of {horizontally
interleaved connecting cells} in $\PPP_{k,m}$. 
At least $\ceil{\alpha m}$ of the points are leaves, and
at least $\ceil{\beta m}+1$ of the strips are empty.
Note first that changing a non-leaf to a leaf can only increase the number of ways of performing the interleaving.
So, for a lower bound, we may assume there are exactly $\ceil{\alpha m}$ leaves.
Note also that, since $\alpha>1/2$, an equitable distribution of leaves among the strips allocates at least one leaf to each strip.
Hence, any increase in the number of empty strips can only make the distribution less equitable.
So, for a lower bound, we may assume there are exactly $\ceil{\beta m}+1$ empty strips.

With these assumptions, given $\alpha$ in the interval $[11/20,5/9)$, $\beta$ in the interval $[7/40,5/27)$ and $m\geqs32$,
an equitable distribution of the leaves among the non-empty strips consists of
\vspace{-9pt}\begin{itemize}\itemsep0pt
  \item $e_0(m) = \ceil{\beta m}+1$ ~empty strips,
  \item $e_2(m) = 3m-4\ceil{\alpha m}-3\ceil{\beta m}$ ~two-leaf strips, and
  \item $e_3(m) = 3\ceil{\alpha m}+2\ceil{\beta m}-2m$ ~three-leaf strips.
\end{itemize}\vspace{-9pt}
The expressions for $e_2(m)$ and $e_3(m)$ are the solutions of the equations
\begin{align*}
    2\+e_2(m) + 3\+e_3(m) & \;=\; \ceil{\alpha m} , \\
    e_0(m) + e_2(m) + e_3(m) & \;=\; m - \ceil{\alpha m} + 1 ,
\end{align*}
for the total number of leaves and the total number of strips, respectively.
{The bounds on $\alpha$, $\beta$ and $m$ ensure that each of the $e_j(m)$ is nonnegative.}

Thus, since the number of components in each connecting cell is exactly $c_m$,
\begin{equation}\label{eqJm}
J_m(z) \;=\;
\big[q^{c_m}\big] \!
\left(
H(z,q)^{e_0(m)} \, H_2(z,q)^{e_2(m)} \, H_3(z,q)^{e_3(m)}
\right)
\end{equation}
is a lower bound for the
generating function of
horizontally interleaved connecting cells in $\PPP_{k,m}$.

To understand the asymptotics of $J_m(z)$ for large $m$, we use the following general result, concerning the exponential growth rate of
combinatorial objects whose generating function has coefficients of the form
$
\big[x^{(\kappa + o(1)) n}\big]\prod_{j=1}^r F_j(x)^{(\alpha_j + o(1)) n}
$, for some fixed $\alpha_1,\ldots,\alpha_r$ and $\kappa$.

\begin{lemma}
Let $\alpha_1,\ldots,\alpha_r$ and $\kappa$ be positive constants.
For each $j\in[r]$, let $F_j(x)$ be a power series with radius of convergence~$\rho_j$.
For each $j$, suppose that $a_{j,1},a_{j,2},\ldots$ is a sequence of positive integers such that
$
\liminfty{a_{j,n}}/n=\alpha_j,
$
and that there is some positive $x_0$, smaller than every $\rho_j$, satisfying
\[
x_0 \sum_{j=1}^r \alpha_j \frac{F_j'(x_0)}{F_j(x_0)} \;=\; \kappa.
\]
Then there exists a sequence of positive integers $c_1,c_2,\ldots$ such that
$
\liminfty{c_n}/n=\kappa
$,
for which
\[
\liminfty \Big( [x^{c_n}] \prod_{j=1}^r F_j(x)^{a_{j,n}} \Big)^{\!1/n} \;=\; x_0^{-\kappa}\+\prod_{j=1}^r F_j(x_0)^{\alpha_j} .
\]
\end{lemma}

This lemma is rather easier to understand and its proof easier to follow when $r=1$. Unfortunately, we need the more general version.

\begin{proof}
For each $j$, define the probability generating function
\[
G_j(x) \;=\; \frac{F_j(x_0\+x)}{F_j(x_0)}.
\]
This definition is valid because $x_0<\rho_j$.

The corresponding expected value is
$
\mu_j = G_j'(1) =
x_0\+F_j'(x_0)/F_j(x_0)
$,
so
\[
\sum_{j=1}^r {\alpha_j\+\mu_j} \;=\; {\kappa}.
\]
For each $j$,
let
$X_j$ be a random variable with probability generating function $G_j$.
For each $n>0$, let $Y_{n}$ be the random variable defined by adding $a_{j,n}$ independent samples from $X_j$ for each $j$.
Then the expected value $\lambda_{n}$ of $Y_{n}$ is given by
\[
\lambda_{n} \;=\; \sum_{j=1}^r a_{j,n}\+\mu_j.
\]
Moreover, it follows from the law of large numbers that if $\veps>0$, then the probability $p_{\veps,n}$ that $Y_{n}$ lies in the interval $(\lambda_{n}(1-\veps),\lambda_{n}(1+\veps))$ converges to 1 as $n$ tends to infinity.
In terms of generating functions, this means
\begin{equation}\label{eqLim}
\liminfty\sum_{c\in(\lambda_{n}(1-\veps),\lambda_{n}(1+\veps))}[x^{c}] \prod_{j=1}^r G_j(x)^{a_{j,n}} \;=\; \liminfty p_{\veps,n} \;=\ 1.
\end{equation}
For each pair $\veps,n$, let $c(\veps,n)$ be the value in the interval $(\lambda_{n}(1-\veps),\lambda_{n}(1+\veps))$ which maximises
\[
[x^{c(\veps,n)}] \prod_{j=1}^r G_j(x)^{a_{j,n}} .
\]
Then, by~\eqref{eqLim}, we have
\[
\liminfinfty ~2\+\veps\+\lambda_{n}\+[x^{c(\veps,n)}] \prod_{j=1}^r G_j(x)^{a_{j,n}} \;\geqs\; 1.
\]
It follows that
\[
\liminfty\Big([x^{c(\veps,n)}] \prod_{j=1}^r G_j(x)^{a_{j,n}}\Big)^{\!1/n} \;=\; 1.
\]
Therefore, we can choose a sequence $c_{1},c_{2},\ldots$ by setting $c_{n}=c(\veps_{n},n)$ in such a way that
\[
\liminfty\veps_{n}\;=\;0~~~~~~\text{and}~~~~~~\liminfty\Big([x^{c_n}] \prod_{j=1}^r G_j(x)^{a_{j,n}}\Big)^{\!1/n} \;=\; 1.
\]
We now show that this sequence satisfies the desired properties. First note that $c_{n}$ lies in the interval $(\lambda_{n}(1-\veps_{n}),\lambda_{n}(1+\veps_{n}))$, so the ratio $c_{n}/\lambda_{n}$ converges to 1.
Moreover,
\[
\liminfty{\lambda_{n}}/{n} \;=\; \liminfty \sum_{j=1}^r a_{j,n}\+\mu_j / n \;=\; \sum_{j=1}^r {\alpha_j\+\mu_j} \;=\; \kappa.
\]
Hence, the ratio $c_{n}/n$ converges to $\kappa$.
Finally,
\begin{align*}
\liminfty\Big([x^{c_{n}}] \prod_{j=1}^r F_j(x)^{a_{j,n}}\Big)^{\!1/n}& \;=\; \liminfty\Big([x^{c_{n}}] \prod_{j=1}^r G_j(x/x_0)^{a_{j,n}} \+ \prod_{j=1}^r F_j(x_0)^{a_{j,n}}\Big)^{\!1/n} \\
& \;=\; x_0^{-\kappa} \+ \liminfty\Big([x^{c_n}] \prod_{j=1}^r G_j(x)^{a_{j,n}}\Big)^{\!1/n} \+ \prod_{j=1}^r F_j(x_0)^{\alpha_j}  \\
& \;=\; x_0^{-\kappa} \+ \prod_{j=1}^r F_j(x_0)^{\alpha_j}. \qedhere
\end{align*}
\end{proof}

Let us apply this lemma to $J_m(z)$, as defined in \eqref{eqJm}.
For any fixed $z_0$, there exists a sequence of positive integers $c_1,c_2,\ldots$ such that
$
\liminfty[m]{c_m}/m=\kappa
$, for which
\begin{equation}\label{eqlimJmz}
\liminfty[m] J_m(z_0)^{1/m} \;=\; q_0^{-\kappa} \,
H(z_0,q_0)^\beta \,
H_2(z_0,q_0)^{3-4\alpha-3\beta} \,
H_3(z_0,q_0)^{3\alpha+2\beta-2} ,
\end{equation}
where $q_0=q_0(z_0)$ satisfies
\begin{equation}\label{eqq0}
      \beta              \!\left.\frac{\frac{d}{dq}  H(z_0,q)}{  H(z_0,q)}\right|_{q=q_0}
  +\: (3-4\alpha-3\beta) \!\left.\frac{\frac{d}{dq}H_2(z_0,q)}{H_2(z_0,q)}\right|_{q=q_0}
  +\: (3\alpha+2\beta-2) \!\left.\frac{\frac{d}{dq}H_3(z_0,q)}{H_3(z_0,q)}\right|_{q=q_0}
\;=\; \frac{\kappa}{q_0}
,
\end{equation}
as long as $q_0$ is less than the radius of convergence in $q$ of the $H_j(z_0,q)$.
Note that each $H_j(z,q)$ can be determined explicitly from the definitions in~\eqref{eqHzq}, \eqref{eqHjzq} and~\eqref{eqOmega}.

%
%
%
%
\subsection{Enumerating the refined staircase}

The first~$6k+2$ cells of the staircase consist of a total of $k+1$ vertical dominoes, each in $\B^{\alpha,\beta}_m$,
a total of $k$ horizontal dominoes, each in $\B_{\ceil{\gamma m}}$,
and $2k$ connecting cells.
Thus, for sufficiently large $m$, the generating function for $\PPP_{k,m}$ is bounded below by
\[
F_{k,m}(z) \;=\;
\Big|\B^{\alpha,\beta}_m\Big|^{k+1} z^{2m(k+1)} \,\,
\Big|\B_{\ceil{\gamma m}}\Big|^k z^{2\ceil{\gamma m}k} \,\,
J_m(z)^{2k} \,\,
\binom{\ceil{\gamma m}+c_m}{c_m}^{\!2k} ,
\]
where the final
binomial coefficient counts the number of possible ways of interleaving the $\ceil{\gamma m}$ points in a horizontal domino
cell with the $c_m$ components in a vertically adjacent connecting cell.

Let $A(z)$ be the generating function for $\av(1324)$, and
for each $k$, let $A_k(z)$ be the generating function for
the set of 1324-avoiding gridded permutations 
in the first $6k+2$ cells of the (original) staircase.
Thus, for any fixed $k$ and $m$, and all $n$,
\[
[z^n]F_{k,m}(z) \;\leqs\; [z^n]A_{k}(z) \;\leqs\; \binom{n+6k+1}{6k+1}[z^n]A(z).
\]
So, since the binomial coefficient is a polynomial in $n$, it follows from the second inequality that the radius of convergence of $A_k(z)$ is at least that of $A(z)$.

Hence,
for any $k$, and any fixed $z_0$ within the radius of convergence of $A(z)$, the value of $F_{k,m}(z_0)$ is bounded above by
$A_k(z_0)$ for every $m$.
So
$\limsupinfty[m]F_{k,m}(z_0)^{1/m}\leqs1$,
and as a consequence,
\[
\liminfty[k]\big(\limsupinfty[m]F_{k,m}(z_0)^{1/m}\big)^{\!1/2k} \;\leqs\; 1 .
\]
By Propositions~\ref{propBalancedDominoes} and~\ref{propBalancedDominoes2}, equation~\eqref{eqlimJmz} and Stirling's approximation,
the left side of this inequality is equal to
\[
G(z_0) \;=\;
\left(\frac{27z_0}{4}\right)^{\!1+\gamma} \+
q_0^{-\kappa} \,
H(z_0,q_0)^\beta \,
H_2(z_0,q_0)^{3-4\alpha-3\beta} \,
H_3(z_0,q_0)^{3\alpha+2\beta-2} \,
\frac{(\gamma+\kappa)^{\gamma+\kappa}}{\gamma^{\gamma}\kappa^{\kappa}}
,
\]
for some appropriate sequence $c_1,c_2,\ldots$,
where $q_0$ is defined by~\eqref{eqq0}.

To prove Theorem~\ref{thmLowerBound2}, it now suffices to find suitable values of $\alpha$, $\beta$, $\gamma$, $\kappa$ and $z_0$,
for which $G(z_0)>1$ and
such that $q_0$
satisfying~\eqref{eqq0}
is less than the radius of convergence in $q$ of the $H_j(z_0,q)$.
Any such $z_0$ lies outside the radius of convergence of $A(z)$
and so $1/z_0$ is a lower bound on the growth rate of $\av(1324)$.
We thus seek $z_0$ as small as possible.

Using $\alpha=5/9-10^{-8}$, 
$\beta=5/27-10^{-8}$, 
$\gamma \approx 0.951509$ and
$\kappa \approx 0.496339$, we may take
the value of $z_0$ to be approximately 0.097361383.
Then $q_0 \approx 2.917054$ and
the radius of convergence of the $H_j(z_0,q)$ is about 9.15, so $q_0$ is in the required range, and $G(z_0)>1$.
Therefore $1/z_0\approx 10.271012$ is a lower bound on the growth rate of $\av(1324)$.\footnote{An alternative approach to analysing the refined staircase suggests that we could take the lower bound to be an algebraic number with a minimal polynomial of degree 104, whose
value is approximately 10.27101292824530.}

%
%
%
%
\subsection{Improving the lower bound further}

How might this result be improved?
Firstly, if we determined the expected proportion of $k$-leaf strips for $k\geqs1$, and established that their distribution was concentrated, then that would affect the optimal distribution of points between the strips, leading to a better bound. It is possible to modify the functional equation for dominoes to record $k$-leaf strips, for any $k$, but the result is complicated and it has not been possible to analyse the result, even for $k=1$.

Secondly, as mentioned at the beginning of Section~\ref{secLowerBound2},
we could relax our construction to permit leaves in vertically adjacent domino cells to be positioned arbitrarily,
like the leaves in horizontally adjacent domino cells are.
Due to the complex interaction between the interleaving of points in two directions, we have not been able to determine a lower bound for the number of possibilities.
It seems likely that the
one-dimensional solution
in which leaves are distributed equitably between the strips
does not carry over to interleaving in two directions.

Finally, 
if we established (a lower bound on) the growth rate of
the set of permutations gridded in the first three cells of the staircase, then we could
decompose the staircase into three-celled \emph{trominoes} to yield a new bound.
However, enumerating
trominoes seems to require some new ideas.

%
%
%
%
%
%
%
%
%
%
%
%
\section*{Acknowledgements}

For a while, Theorem~\ref{thm-twocell} was just a conjecture based on numerical evidence.
We are grateful to Vince Vatter, who was asked by the first two authors to present this conjecture
in the problem session at the Permutation Patterns 2016 conference, which they were unable to attend.

%
%
%
%
%
%
%
%
%
%
%
%
\bibliographystyle{acm}
\bibliography{./refs-staircase}

\begin{thebibliography}{10}

\bibitem{albert:on-the-wilf-sta:}
{\sc Albert, M.~H., Elder, M., Rechnitzer, A., Westcott, P., and Zabrocki, M.}
\newblock On the {S}tanley--{W}ilf limit of $4231$-avoiding permutations and a
  conjecture of {A}rratia.
\newblock {\em Adv. in Appl. Math. 36}, 2 (2006), 95--105.

\bibitem{albert:on-the-growth-of-merges:}
{\sc Albert, M.~H., Pantone, J., and Vatter, V.}
\newblock On the growth of merges and staircases of permutation classes.
\newblock {\em Rocky Mountain J. Math.\/} (forthcoming).

\bibitem{BevanAv1324GR}
{\sc Bevan, D.}
\newblock Permutations avoiding 1324 and patterns in {{\L}}ukasiewicz paths.
\newblock {\em J.~London Math. Soc. 92}, 1 (2015), 105--122.

\bibitem{bona:exact-enumerati:}
{\sc B{\'o}na, M.}
\newblock Exact enumeration of {$1342$}-avoiding permutations: a close link
  with labeled trees and planar maps.
\newblock {\em J.~Combin. Theory Ser.~A 80}, 2 (1997), 257--272.

\bibitem{bona:the-exponential-up:}
{\sc B{\'o}na, M.}
\newblock A simple proof for the exponential upper bound for some tenacious
  patterns.
\newblock {\em Adv. in Appl. Math. 33}, 1 (2004), 192--198.

\bibitem{bona:the-limit-of-a-:}
{\sc B{\'o}na, M.}
\newblock The limit of a {S}tanley--{W}ilf sequence is not always rational, and
  layered patterns beat monotone patterns.
\newblock {\em J.~Combin. Theory Ser.~A 110}, 2 (2005), 223--235.

\bibitem{Bona2012}
{\sc B{\'o}na, M.}
\newblock {\em Combinatorics of Permutations}, second~ed.
\newblock CRC Press, 2012.

\bibitem{bona:a-new-upper-bou:}
{\sc B\'ona, M.}
\newblock A new upper bound for 1324-avoiding permutations.
\newblock {\em Combin. Probab. Comput. 23}, 5 (2014), 717--724.

\bibitem{bona:a-new-record-for:}
{\sc B\'ona, M.}
\newblock A new record for 1324-avoiding permutations.
\newblock {\em Eur. J. Math. 1}, 1 (2015), 198--206.

\bibitem{bousquet-melou:forest-like-perms}
{\sc Bousquet-M{\'e}lou, M., and Butler, S.}
\newblock Forest-like permutations.
\newblock {\em Ann. Comb. 11}, 3--4 (2007), 335--354.

\bibitem{bousquet-melou:poly-eqs}
{\sc Bousquet-M{\'e}lou, M., and Jehanne, A.}
\newblock Polynomial equations with one catalytic variable, algebraic series
  and map enumeration.
\newblock {\em J.~Combin. Theory Ser.~B 96}, 5 (2006), 623--672.

\bibitem{brown:non-sep-planar-maps}
{\sc Brown, W.~G.}
\newblock Enumeration of non-separable planar maps.
\newblock {\em Canad. J. Math. 15\/} (1963), 526--545.

\bibitem{claesson:upper-bounds-fo:}
{\sc Claesson, A., Jel\'inek, V., and Steingr\'imsson, E.}
\newblock Upper bounds for the {S}tanley--{W}ilf limit of 1324 and other
  layered patterns.
\newblock {\em J.~Combin. Theory Ser.~A 119}, 8 (2012), 1680--1691.

\bibitem{conway:on-1324-avoiding:}
{\sc Conway, A.~R., and Guttmann, A.~J.}
\newblock On 1324-avoiding permutations.
\newblock {\em Adv. in Appl. Math. 64\/} (2015), 50--69.

\bibitem{CGZJ2018}
{\sc Conway, A.~R., Guttmann, A.~J., and Zinn-Justin, P.}
\newblock 1324-avoiding permutations revisited.
\newblock {\em Adv. in Appl. Math. 96\/} (2018), 312--333.

\bibitem{DGRS2017}
{\sc Duchi, E., Guerrini, V., Rinaldi, S., and Schaeffer, G.}
\newblock Fighting fish.
\newblock {\em J.~Phys.~A 50}, 2 (2017).

\bibitem{DGRS2016}
{\sc Duchi, E., Guerrini, V., Rinaldi, S., and Schaeffer, G.}
\newblock Fighting fish: enumerative properties.
\newblock {\em S\'em. Lothar. Combin. 78B\/} (2017), ~Art. 43, 12~pp.

\bibitem{Egge2015}
{\sc Egge, E.~S.}
\newblock Defying {G}od: the {S}tanley-{W}ilf conjecture, {S}tanley-{W}ilf
  limits, and a two-generation explosion of combinatorics.
\newblock In {\em A century of advancing mathematics}. Math. Assoc. America,
  Washington, DC, 2015, pp.~65--82.

\bibitem{Fang2017}
{\sc Fang, W.}
\newblock Fighting fish and two-stack sortable permutations.
\newblock {\em S\'em. Lothar. Combin. 80B\/} (2018), ~Art. 7, 12~pp.

\bibitem{flajolet:ac}
{\sc Flajolet, P., and Sedgewick, R.}
\newblock {\em Analytic combinatorics}.
\newblock Cambridge University Press, 2009.

\bibitem{GP2017}
{\sc Garrabrant, S., and Pak, I.}
\newblock Words in linear groups, random walks, automata and
  \mbox{{P}-recursiveness}.
\newblock {\em J.~Combin. Algebra 1}, 2 (2017), 127--144.

\bibitem{gessel:symmetric-funct:}
{\sc Gessel, I.~M.}
\newblock Symmetric functions and {P}-recursiveness.
\newblock {\em J.~Combin. Theory Ser.~A 53}, 2 (1990), 257--285.

\bibitem{Kitaev2011}
{\sc Kitaev, S.}
\newblock {\em Patterns in Permutations and Words}.
\newblock Springer, 2011.

\bibitem{klazar:bell-numbers}
{\sc Klazar, M.}
\newblock Bell numbers, their relatives, and algebraic differential equations.
\newblock {\em J.~Combin. Theory Ser.~A 102}, 1 (2003), 63--87.

\bibitem{ML2010}
{\sc Madras, N., and Liu, H.}
\newblock Random pattern-avoiding permutations.
\newblock In {\em Algorithmic Probability and Combinatorics}. Amer. Math. Soc.,
  2010, pp.~173--194.

\bibitem{marcus:excluded-permut:}
{\sc Marcus, A., and Tardos, G.}
\newblock Excluded permutation matrices and the {S}tanley--{W}ilf conjecture.
\newblock {\em J.~Combin. Theory Ser.~A 107}, 1 (2004), 153--160.

\bibitem{noy:noncrossing-trees}
{\sc Noy, M.}
\newblock Enumeration of noncrossing trees on a circle.
\newblock {\em Discrete Math. 180}, 1 (1998), 301--313.

\bibitem{StanleyCatalan}
{\sc Stanley, R.~P.}
\newblock {\em Catalan Numbers}.
\newblock Cambridge Univ. Press, 2015.

\bibitem{Steingrimsson2013}
{\sc Steingr\'imsson, E.}
\newblock Some open problems on permutation patterns.
\newblock In {\em Surveys in Combinatorics 2013}, vol.~409 of {\em London Math.
  Soc. Lecture Note Ser.} Cambridge Univ. Press, 2013, pp.~239--263.

\bibitem{OEIS}
{\sc {The OEIS Foundation Inc.}}
\newblock The {O}n-{L}ine {E}ncyclopedia of {I}nteger {S}equences.
\newblock Published electronically at
  \href{https://oeis.org}{https:/$\!$/oeis.org}.

\bibitem{VatterSurvey}
{\sc Vatter, V.}
\newblock Permutation classes.
\newblock In {\em The Handbook of Enumerative Combinatorics}, M.~B{\'o}na, Ed.
  CRC Press, 2015, pp.~753--833.

\bibitem{Mathematica}
{\sc {Wolfram Research, Inc.}}
\newblock Mathematica. {V}ersion 10.3.
\newblock
  \href{http://www.wolfram.com/mathematica}{www.wolfram.com/mathematica}, 2015.

\bibitem{zeilberger:west}
{\sc Zeilberger, D.}
\newblock A proof of {J}ulian {W}est's conjecture that the number of
  two-stack-sortable permutations of length $n$ is $2(3n)!/((n + 1)!(2n +
  1)!)$.
\newblock {\em Discrete Math. 102}, 1 (1992), 85--93.

\end{thebibliography}

\end{document}